\documentclass[12pt]{article}
\usepackage{amsthm}
\usepackage{amsmath}
\usepackage{amsfonts}
\usepackage{graphicx}
\usepackage{latexsym}
\usepackage{amssymb}
\usepackage{color}

\newcommand{\field}[1]{\mathbb{#1}}
\newcommand{\B}{\field{B}}
\newcommand{\C}{\field{C}}
\newcommand{\Z}{\field{Z}}

\newcommand{\cB}{{\cal B}}

\newcommand{\cS}{{\cal S}}

\newcommand{\cT}{{\cal T}}
\newcommand{\cP}{{\cal P}}

\newcommand{\cR}{{\cal R}}

\newtheorem{theorem}{Theorem}
\newtheorem{lemma}{Lemma}
\newtheorem{example}{Example}
\newtheorem{remark}{Remark}
\newtheorem{corollary}{Corollary}

\begin{document}

\bibliographystyle{plain}

\title{
\begin{center}
Large Sets with Multiplicity
\end{center}
}
\author{
{\sc Tuvi Etzion}\thanks{Department of Computer Science, Technion,
Haifa 3200003, Israel, e-mail: {\tt etzion@cs.technion.ac.il}.} \and
{\sc Junling Zhou}\thanks{Department of Mathematics,
Beijing Jiaotong University, Beijing, China,
e-mail: {\tt jlzhou@bjtu.edu.cn}.}}

\maketitle

\begin{abstract}
Large sets of combinatorial designs has always been a fascinating topic in
design theory. These designs form a partition of the whole space into
combinatorial designs with the same parameters. In particular, a large set of block designs,
whose blocks are of size $k$ taken from an $n$-set, is a partition of all the $k$-subsets
of the $n$-set into disjoint copies of block designs, defined on the $n$-set, and with the same parameters.
The current most intriguing question in this direction
is whether large sets of Steiner quadruple systems exist and to provide explicit constructions
for those parameters for which they exist. In view of its difficulty no one ever presented an explicit construction
even for one nontrivial order. Hence, we seek for related generalizations.
As generalizations, to the existence question of large sets, we consider two related questions.
The first one to provide constructions for sets on Steiner systems in which each block (quadruple or
a $k$-subset) is contained in exactly $\mu$ systems.
The second question is to provide constructions
for large set of H-designs (mainly for quadruples, but also for larger block size).
We prove the existence of such systems for many parameters using orthogonal arrays,
perpendicular arrays, ordered designs, sets of permutations, and one-factorizations of the complete graph.
\end{abstract}

\vspace{0.5cm}


\footnotetext[1] {This research was supported in part by the
111 Project of China (B16002) and in part by the NSFC grants 11571034 and 11971053.
Part of the research was performed during a visit of T. Etzion to Beijing Jiaotong University.
He expresses sincere thanks to the 111 Project of China (B16002) for its support and to the Department
of Mathematics at Beijing Jiaotong University for their kind hospitality.

T. Etzion was also supported in part by the Bernard Elkin Chair in Computer Science.}

\newpage
\section{Introduction}
\label{sec:introduction}

A \emph{Steiner system of order $n$}, S$(t,k,n)$, is a pair $(Q,B)$, where $Q$ is an $n$-set
(whose elements are called \emph{points}) and $B$ is a collection
of $k$-subsets (called \emph{blocks}) of $Q$, such that each $t$-subset of $Q$ is contained
in exactly one block of $B$. A \emph{large set} of Steiner systems S$(t,k,n)$, on an $n$-set $Q$, is a partition of all $k$-subsets
of $Q$ into Steiner systems S$(t,k,n)$.

The interest in large sets is from block design point of view and also from graph theory point of view.
A large set of Steiner systems S$(1,k,n)$ is equivalent to a partition of the
$k$-uniform complete hypergraph into disjoint perfect matchings. If $k=2$ this large set
is known as \emph{one-factorization} of the complete graph $K_n$. A comprehensive discussion
on these one-factorizations are given in~\cite{Wal97}. Peltesohn~\cite{Pel36} solved
the problem for $k=3$. A solution for $k \geq 4$ was given in the celebrated work of
Baranayai~\cite{Bar75} who proved the existence of such large sets using a network flow.

 A Steiner system S$(2,3,n)$ is known as Steiner triple
system and the corresponding large set is known to exist for every admissible $n \equiv 1$ or $3 ~(\text{mod}~6)$, where $n \neq 7$.
It was first proved by Lu~\cite{Lu83,Lu84}, who left six open cases which were solved by Teirlinck~\cite{Tei91}.
An alternative shorter proof was given later by Ji~\cite{Ji05}.
The next interesting case is for Steiner system S$(3,4,n)$, which is also called a \emph{Steiner quadruple system}
and denoted by SQS$(n)$. A construction of large set of SQS$(n)$ is known only for the
trivial case when $n=4$. For larger $n$,
the only known result is the existence proof of Keevash~\cite{Kee18} who proved that
a large set of Steiner systems S$(t,k,n)$ exists if $n$ is large enough and satisfies
some necessary conditions (this $n$ is beyond our imagination).
The proof is nonconstructive and hence it does not throw any light on any explicit construction for these values of $n$.
It either does not provide any indication on the existence of such systems for smaller values of $n$ (which can be very large).

In the absence of known constructions for large sets of SQS$(n)$ the research can be done in two
different directions. One direction is to find the maximum number of pairwise disjoint SQS$(n)$. The best results in this
direction can be found in~\cite{EtHa91,EtZh20}. A second direction is to define large sets with multiplicity.
A \emph{large set of S$(t,k,n)$ with multiplicity $\mu$}, denoted by LS$(t,k,n;\mu)$,
is a set of Steiner systems S$(t,k,n)$ on an $n$-set~$Q$, such that each $k$-subset of~$Q$ is contained
in exactly $\mu$ systems. The goal is to find such large set for any given positive integer $\mu$, where
LS$(t,k,n;1)$ is a large set which implies LS$(t,k,n;\mu)$ for any $\mu \geq 1$.
Clearly, an LS$(t,k,n,\mu)$ consists of $\mu \binom{n-t}{k-t}$ Steiner systems S$(t,k,n)$.
Large sets with multiplicity were considered in~\cite{Etz96}, where it is proved that
LS$(3,4,2^r;\mu)$ exists, for any $r \geq 3$ and $\mu \geq 2$. Such large sets with multiplicity implies
the existence of another family of large sets (with multiplicity one), namely,
large sets of H-designs, which are interesting designs for themselves. They have applications in threshold schemes~\cite{Etz96}
and in quantum jump codes~\cite{ZhCh17}.

The goals of this paper are to construct large sets with multiplicity and large sets of H-designs.
The rest of this paper is organized as follows. In Section~\ref{sec:preliminary} we present the basic concepts
required for our expositions. These include the definition of an H-design and the connection between
a large set with multiplicity and large set of H-designs. Other concepts include orthogonal arrays,
one-factorizations, and arrays of permutations.
In Section~\ref{sec:large} we present some basic constructions for large sets with multiplicity and
for large sets of H-designs. The number of groups in these constructions of H-designs
is small and the same is true for the number of points in the large sets with multiplicity.
Section~\ref{sec:large} is devoted to large sets with multiplicity and large sets of H-designs with small parameters.
In Section~\ref{sec:small} large sets of H-designs
are obtained recursively from an initial large set of H-designs with small parameters.
In Section~\ref{sec:large_permute} large sets with multiplicity are constructed using
some structure of Steiner systems and sets of permutations.
They yield large sets of H-designs with small parameters. In Section~\ref{sec:review} we present
a few well-known constructions of pairwise disjoint SQS$(n)$ which will be adapted and used
for our constructions of large sets with multiplicity of Steiner quadruple systems.
These constructions contain doubling and quadrupling constructions.
In Section~\ref{sec:five} and Section~\ref{sec:mainC} we present our main construction for
large sets with multiplicity of Steiner quadruple systems. In Section~\ref{sec:five} a quadrupling construction
for large sets with multiplicity are presented. In Section~\ref{sec:mainC} the quadrupling construction
is generalized for multiplication by $2^m$ instead of multiplication by 4.
In Section~\ref{sec:conclusion}
we summarize our work and suggest directions for future research.

\section{Preliminaries}
\label{sec:preliminary}

This section is devoted to define several concepts which are important in our exposition.
In Section~\ref{sec:Hdesigns} we define the concepts of H-designs and large sets of H-designs
whose construction is one of the goals of our work. In Section~\ref{sec:OA} we define the concepts of orthogonal
arrays and large sets of orthogonal arrays. Finally, we prove a connection between
large set with multiplicity, large set of orthogonal arrays, and large set of H-designs.
In Section~\ref{sec:oneF} we define a design used in many constructions of block designs,
namely, one-factorization.
In Section~\ref{sec:perm_arrays} we consider permutations and arrays of permutations such as
Latin squares, ordered designs, and perpendicular arrays. Finally, in Section~\ref{sec:config}
we define the concept of configurations which enables us to categorize the different blocks
of a design after some partition of the point set is made.

\subsection{H-designs}
\label{sec:Hdesigns}

Large sets with multiplicity have their own interest, but they are also important in
constructions for large sets of H-designs, which are large sets of Steiner systems ``with holes"~\cite{Tei93}.
An \emph{H-design} H$(n,g,k,t)$ is a triple $(Q,G,B)$, which satisfies the following
properties:

\begin{enumerate}
\item $Q$ is a set with $ng$ points.

\item $G$ is a partition of $Q$ into $n$ subsets (called \emph{groups}), each one
with $g$ points.

\item $B$ is a set of $k$-subsets of $Q$ (called \emph{blocks}), such that a group and a block
contain at most one common point, and any $t$ points from any $t$ distinct groups occur in exactly one block.
\end{enumerate}

To simplify the constructions in the sequel, we will assume that for a given H-design H$(n,g,k,t)$,
the elements of the $(ng)$-set $Q$ are ordered and in a block $\{ x_1,x_2,\ldots,x_k\}$ the elements
are ordered, i.e., $x_1 < x_2 < \cdots < x_k$, by the order of $Q$.

H-designs were defined and used first in~\cite{HMM,Mil90}.
Necessary and sufficient conditions for the existence of H-designs for which $(k,t)=(4,3)$
were proved in~\cite{Ji09,Ji19,Mil90}.

\begin{theorem}
\label{thm:necess}
The necessary and sufficient conditions for the existence of an H$(n, g, 4, 3)$ are $gn\equiv 0$ {\rm (mod 2)},
$g(n-1)(n -2)\equiv 0$ {\rm (mod 3)}, $n\geq 4$, and $(n,g)\neq (5,2)$.
\end{theorem}

A \emph{large set of H-designs}, denoted by LH$(n,g,k,t)$, is a partition, of all the $k$-subsets
from the $ng$ points of $Q$ taken from any $k$ distinct groups, into pairwise
disjoint H-designs H$(n,g,k,t)$. The number of H-designs in a large set is
calculated in the following lemma.

\begin{lemma}
\label{lem:LSH_size}
The size of an LH$(n,g,k,t)$ is $\binom{n-t}{k-t} g^{k-t}$.
\end{lemma}
\begin{proof}
An H-design H$(n,g,k,t)$ contains $\frac{\binom{n}{t}}{\binom{k}{t}} g^t$ blocks and
the number of $k$-subsets in $Q$, where each $k$-subset contains at most one
element from each group, is $\binom{n}{k} g^k$. Hence, the number of
H-designs in an LH$(n,g,k,t)$ is
$\binom{n}{k} g^k  {\Huge / }  {\frac{\binom{n}{t}}{\binom{k}{t}} g^t}=\frac{(n-t)!}{(n-k)!(k-t)!} g^{k-t} = \binom{n-t}{k-t} g^{k-t}$.
\end{proof}

An H-design H$(n, g, 3, 2)$ is usually called a \emph{group divisible design} of type
$g^n$ and denoted by GDD$(g^n)$. The large sets of
disjoint group divisible designs were first studied because of their connection with
perfect threshold schemes~\cite{ScSt89,StVa88}.
Combining the existence result of large sets of Steiner triple systems and much work
on large sets of GDDs by Chen et al~\cite{CLS92} and Teirlinck~\cite{Tei93}, Lei~\cite{Lei97} finally
established that there exists a large set of GDD$(g^n)$ if and only if $n(n-1)g^2 \equiv 0$
(mod 6), $(n-1)g\equiv 0$ (mod 2), and $(g,n) \neq (1,7)$.

\subsection{Orthogonal Arrays}
\label{sec:OA}

An \emph{orthogonal array} OA$(t,k,n)$ is an $n^t \times k$ matrix $C$, with entries from $\Z_n$,
such that any submatrix generated by any $t$ columns of $C$ contains each ordered $t$-tuple from $\Z_n$ exactly once as a row.
A \emph{large set of orthogonal arrays} LOA$(t,k,n)$ is a set of $n^{k-t}$ orthogonal arrays OA$(t,k,n)$,
such that each vector of length $k$ over $\Z_n$ occurs as a row in exactly one of the orthogonal arrays.

\begin{theorem}
\label{thm:OAtoLOA}
If there exists an OA$(t,k,n)$, then there exists an LOA$(t,k,n)$.
\end{theorem}
\begin{proof}
Let $C$ be an orthogonal array OA$(t,k,n)$. Let $X$ be the set of $n^{k-t}$ vectors of length~$k$ over $\Z_n$
such that the last $t$ entries in all the vectors of $X$ are zeroes.
For each $x \in X$ define the set $C_x$ as follows:
$$
C_x \triangleq x + C \triangleq \{ x+c ~:~ c \in C \}.
$$
Clearly, $C_x$ is also an orthogonal array OA$(t,k,n)$.

Assume now that there exist two words $x_1,x_2 \in X$ and two words $c_1,c_2 \in C$ such that
$x_1 + c_1 = x_2 + c_2$. Since the last $t$ entries in $x_1$ and $x_2$ are zeroes. This implies
that the last $t$ entries of $c_1$ and $c_2$ are equal. Since $C$ is an orthogonal array OA$(t,k,n)$ and the
last $t$ entries of $c_1$ and $c_2$ are equal, it follows that $c_1=c_2$.
Hence, we also have $x_1 = x_2$ and therefore $\{ C_x ~:~ x \in X \}$ is an LOA$(t,k,n)$.
\end{proof}

By Theorem~\ref{thm:OAtoLOA} the existence of an OA$(t,k,n)$ implies the existence
of an LOA$(t,k,n)$. In the following results we will use these orthogonal arrays to form
large sets of H-designs. The related constructions will require the existence of other large sets
of H-designs with smaller group size or the existence of some large sets with multiplicity.

\begin{theorem}
\label{thm:expand}
If there exists an LH$(n,g,k,t)$ and an
OA$(t,k,u)$, then there exists a large set of H-designs LH$(n,gu,k,t)$.
\end{theorem}
\begin{proof}
Let $A_1,A_2,\ldots,A_\delta$ be an LH$(n,g,k,t)$ defined on an $(ng)$-set $Q$
(taken as $\Z_n \times \Z_g$ with group set $\{\{i\}\times \Z_g ~:~ i \in \Z_n\}$),
where $\delta=\binom{n-t}{k-t} g^{k-t}$. By Theorem~\ref{thm:OAtoLOA} the existence
of OA$(t,k,u)$ implies the existence of an LOA$(t,k,u)$. Let
$B_1,B_2,\ldots,B_{\gamma}$ be an LOA$(t,k,u)$ defined on $\Z_u$, where $\gamma=u^{k-t}$.
Define $\delta \gamma$ sets $C_{i,j}$, $1\le i\le \delta$, $1\le j\leq \gamma$,
where $C_{i,j}$ consists of $\frac{\binom{n}{t}}{\binom{k}{t}}(gu)^t$ blocks.
For each $\{(x_1,y_1),(x_2,y_2),\ldots,(x_k,y_k)\}\in A_i$, where $(x_\ell,y_\ell) \in \Z_n \times \Z_g$,
and each $(b_1,b_2,\ldots,b_k)\in B_j$ the following block is defined:
$$
\{(x_1,(y_1,b_1)),(x_2,(y_2,b_2)),\ldots,(x_k,(y_k,b_k))\}~.
$$
It is easy to verify that each $C_{i,j}$ forms an H-design with group set $\{ \{ i \} \times (\Z_g \times \Z_u) ~:~ i \in \Z_n \}$.
By Lemma~\ref{lem:LSH_size}, the number of H-designs in LH$(n,gu,k,t)$ is $\binom{n-t}{k-t} (gu)^{k-t} = \delta \gamma$ and
the size of $\C \triangleq \{ C_{i,j} ~:~ 1 \leq i \leq \delta,~ 1 \leq j \leq \gamma \}$ is $\delta \gamma$. Therefore,
to prove that $\C$ forms a large set of H-designs LH$(n,gu,k,t)$, it is
sufficient to show that each $k$-subset of $\Z_n \times (\Z_g \times \Z_u)$, meeting each group in at
most one point, is contained in one H-designs from\linebreak
$\{ C_{i,j} ~:~ 1 \leq i \leq \delta,~ 1 \leq j \leq \gamma \}$.

Let $Z=\{ (x_1,(y_1,b_1)),(x_2,(y_2,b_2)),\ldots,(x_k,(y_k,b_k)) \}$,
where $(x_i,(y_i,b_i)) \in \Z_n \times (\Z_g \times \Z_u))$ and $x_1 < x_2 < \cdots < x_k$.
The set $Z'=\{ (x_1,y_1),(x_2,y_2),\ldots,(x_k,y_k) \}$ is a $k$-subset of $\Z_n \times \Z_g$
such that $x_1 < x_2 < \cdots < x_k$ and hence $Z'$ is a $k$-subset in a unique H-design, say~$A_i$.
Also, $(b_1,b_2,\ldots,b_k)$ is a row in a unique orthogonal array $B_j$. Therefore,
$Z$ is a $k$-subset of the unique set $C_{i,j}$. Thus, $\{ C_{i,j} ~:~ 1 \leq i \leq \delta,~ 1 \leq j \leq \gamma \}$
is a large set of H-designs LH$(n,gu,k,t)$.
\end{proof}

The next result generalizes a related theorem for LH$(n,g,4,3)$ proved in~\cite{Etz96}.
\begin{theorem}
\label{thm:LSH}
If there exist an OA$(t,k,g)$ and an LS$(t,k,n;g^{k-t})$,
then there exists an LH$(n,g,k,t)$.
\end{theorem}
\begin{proof}
Let $S_1,S_2,\ldots,S_\delta$, $\delta =\binom{n-t}{k-t} g^{k-t}$, be an LS$(t,k,n;g^{k-t})$, on the point set $\Z_n$,
and let $C_1,C_2,\ldots,C_{g^{k-t}}$ be an LOA$(t,k,g)$ implied by
Theorem~\ref{thm:OAtoLOA}. We construct
an LH$(n,g,k,t)$ on the point set $\Z_n \times \Z_g$, i.e. group sets
$\{i\} \times \Z_g$, $i \in \Z_n$, with H-designs $S^*_1,S^*_2,\ldots,S^*_\delta$.
Given any $k$-subset $\{ x_1 , x_2 ,\ldots , x_k \}$ of $\Z_n$ which appears in $g^{k-t}$ distinct Steiner
systems, $S_{i_1},S_{i_2},\ldots,S_{i_{g^{k-t}}}$, we form the following $g^t$ blocks for $S^*_{i_j}$,
$1 \leq j \leq g^{k-t}$,
\begin{equation}
\label{eq:ls_Hd}
\{ (x_1,y_1),(x_2,y_2),\ldots,(x_k,y_k) \},~~ (y_1,y_2,\ldots,y_k) \in C_j~.
\end{equation}
By Lemma~\ref{lem:LSH_size}, the number of H-designs in LH$(n,g,k,t)$ is $\delta$ and hence
to prove that $S^*_1,S^*_2,\ldots,S^*_\delta$ form an LH$(n,g,k,t)$, it is
sufficient to show that each $k$-subset of $\Z_n \times \Z_g$, meeting each group in at
most one point, is contained in one of the H-designs from
$S^*_1,S^*_2,\ldots,S^*_\delta$.

Let $Z=\{ (x_1,y_1),(x_2,y_2),\ldots,(x_k,y_k) \}$, where $(x_i,y_i) \in \Z_n \times \Z_g$ and
${|\{ x_1 , x_2 ,\ldots , x_k \}|=k}$.
Since $X=\{ x_1 , x_2 ,\ldots , x_k \}$ is a $k$-subset of $\Z_n$,
it follows that $X$ is contained in exactly $g^{k-t}$ Steiner systems $S_{i_1},S_{i_2},\ldots,S_{i_{g^{k-t}}}$
of the large set with multiplicity LS$(t,k,n;g^{k-t})$. Since $Y=( y_1 , y_2 ,\ldots , y_k )$ is
a word of length $k$ over $\Z_g$,
it follows that $Y$ is a row of an OA$(t,k,g)$ from the related large set.
It follows by the construction implied by (\ref{eq:ls_Hd}) that $Z$ is contained in one of the constructed blocks.
This completes the proof.
\end{proof}

Theorem~\ref{thm:LSH} is the key for our construction of new large sets of H-designs which
will be discussed in Section~\ref{sec:large}. In~\cite{Etz96} we have the following application
of Theorem~\ref{thm:LSH}.

\begin{theorem}
\label{thm:LHpower 2}
For any integers $g \geq 2$ and $r \geq 2$, there exists an LH$(2^r,g,4,3)$.
\end{theorem}

Theorem~\ref{thm:LHpower 2} was proved in~\cite{Etz96} without the explicit use of orthogonal
arrays. The orthogonal arrays required to obtain this theorem and the ones which will be heavily used
in our exposition are presented in the following trivial well-known theorem.

\begin{theorem}
\label{thm:t_t+1OA}
For any integer $t \geq 2$ and for any given $g \geq 2$, there exists an OA$(t-1,t,g)$.
\end{theorem}
\begin{proof}
Define the following set of $t$-tuples
$$
M \triangleq \{ (x_1,x_2,\ldots, x_t) ~:~  x_i \in \Z_g,~ \sum_{i=1}^t x_i \equiv 0~(\text{mod}~g)   \}~.
$$
Consider the set $M$ as an array with $t$ columns with any order of the rows by the $t$-tuples of $M$.
Clearly, each $(t-1)$-tuple appears exactly once in the projection of $t-1$ coordinates of~$M$,
which implies that the array~$M$ is an OA$(t-1,t,g)$.
\end{proof}
Other orthogonal arrays can be applies of other LS$(t,k,n;\mu)$ presented in our expositions.
These can be found in the extensive literature starting with the book of Raghavarao~\cite{Rag71}.

\subsection{One-Factorizations}
\label{sec:oneF}

A \emph{one-factorization} $F=\{F_0,F_1,\ldots,F_{v-2} \}$ of the complete graph $K_v$, where $v$ is an even positive integer,
is partition of all the edges of $K_v$ into perfect matchings. Each $F_i$, $0 \leq i \leq v-2$, is a perfect matching of $K_v$.
This perfect matching is also called a \emph{one-factor}. As mentioned before a one-factor is a Steiner system $S(1,2,v)$
and a one-factorization is a related large set.
Clearly, a one-factor contains $\frac{v}{2}$ pairs of vertices whose
union is the set of $v$ vertices in $K_v$. One-factorizations were used as building blocks in many constructions
of various block designs.

\subsection{Arrays of Permutations}
\label{sec:perm_arrays}

A permutation acting on the point set of a Steiner system S$(t,k,n)$ yields another Steiner system S$(t,k,n)$.
This is the obvious motivation to use arrays of permutations in constructions of large sets of Steiner systems
with multiplicity and large sets of H-designs.

A $v \times v$ \emph{Latin square} is a $v \times v$ array in which each row and each column
is a permutation of a $v$-set $Q$.
A~Latin square has no $2 \times 2$ subsquares if any
$2 \times 2$ subsquare restricted to two rows and two columns does not form a Latin square~\cite{KLR75}.
Such Latin squares were used in a doubling construction~\cite{Lin77} to form a set of pairwise disjoint SQSs.
The DLS Construction, which will be presented in Section~\ref{sec:review},
is based on the Construction in~\cite{Lin77}, but the constructions which
will be given in Section~\ref{sec:five} and Section~\ref{sec:mainC} do not require such squares.
We will make use of combinatorial designs called \emph{ordered designs} and \emph{perpendicular arrays}.

An \emph{ordered design} OD$_\lambda (k,\ell,n)$ is a $\lambda \cdot \binom{n}{k} \cdot k! \times \ell$
matrix $A$ with entries from an $n$-set, say~$\Z_n$, such that
\begin{enumerate}
\item each row has $\ell$ distinct entries and

\item each submatrix of $A$ which consists of any $k$ columns contains each $k$-tuple of $\Z_n$ exactly $\lambda$ times.
\end{enumerate}

A \emph{perpendicular array} PA$_\lambda (k,\ell,n)$ is a $\lambda \cdot \binom{n}{k} \times \ell$
matrix $A$ with entries from an $n$-set, say~$\Z_n$, such that
\begin{enumerate}
\item each row has $\ell$ distinct entries and

\item each submatrix of $A$ which consists of any $k$ columns contains each $k$-subset of $\Z_n$ exactly $\lambda$ times.
\end{enumerate}

Although probably known before, the first formal definition of perpendicular arrays
is given in~\cite{MSRV80}. Ordered designs were defined first by Teirlinck~\cite{Tei87,Tei88,Tei89}.
Perpendicular arrays have found applications in authentication and secrecy codes~\cite{Sti90}.
This has motivated an extensive research, e.g.~\cite{BBE96,Bie92,Bie95,BiEd94,BiEd99,BiTr91,BiTr91a,GeZh96,GMR88,MSRV80,KWMT91,MaSa05}.
Some of the construction which follows will make use of ordered designs and perpendicular arrays
which can also be regarded as permutation sets.

%

An ordered design $OD_\lambda (k,n,n)$ or a perpendicular array $PA_\lambda (k,n,n)$
consists of a set of permutations from $S_n$, the set of all permutations on an $n$-set.
A set $P \subseteq S_n$ of permutations is (uniformly) \emph{$k$-homogeneous} if it is a $PA_\lambda (k,n,n)$;
it is \emph{$k$-transitive} if it is an $OD_\lambda (k,n,n)$. These are the combinatorial designs
required for the constructions in the sequel.

\vspace{0.2cm}

A few types of ordered designs $OD_\lambda (k,n,n)$ and perpendicular arrays $PA_\lambda (k,n,n)$ are required, e.g.
those with $k=2$ or those with $k>2$ and small $n$. For $k >3$ these permutations sets are required for a
simple construction of an LS$(t,k,n;\mu)$ from a given Steiner system S$(t,k,n)$.
This construction is presented in the following theorem.

\begin{theorem}
\label{thm:perm_LS}
If there exists a Steiner system S$(t,k,n)$ and a perpendicular array PA$_\lambda (k,n,n)$, then there
exist an LS$(t,k,n;\mu)$, where $\mu = \lambda \binom{n}{t} {\Huge /} \binom{k}{t}$.
\end{theorem}
\begin{proof}
By applying all the permutations of a perpendicular array PA$_\lambda (k,n,n)$ on each block of
a Steiner system S$(t,k,n)$, each $k$-subset of the $n$-set is produced exactly $\lambda$ times
and hence the claim follows.
\end{proof}

Unfortunately, if $k > 3$ and $n \geq 25$ there is no subgroup of $S_n$ which forms such a permutation set,
except for the group of all permutation $S_n$ and the alternating group $A_n$
which contains all the even permutation (see~\cite[Theorem 5.2]{Cam95}).
A probabilistic proof for the existence of such a permutation set with a
smaller number of permutations is given~\cite{KLP17}. A probabilistic construction with even a smaller number
of permutations is given in~\cite{FPY15}. In both cases the number of permutations is still too large
and there is no concrete construction, rather than an existence proof or
a construction based on probabilistic arguments. For $n \leq 24$ there are some interesting designs which yield
some interesting large sets. These constructions will be considered in Section~\ref{sec:large}
when large sets with small parameters will be discussed.

For $k=2$, PA$_\lambda (2,n,n)$ will be used in our main constructions instead of Latin squares
which were used in other constructions~\cite{Etz96,EtZh20,Lin77} (also OD$_\lambda (2,n,n)$ can be
used for this purpose). Some parameters for these designs are given in the following theorem.

\begin{theorem}
\label{thm:PA_pairs}
The following perpendicular arrays exist:
\begin{enumerate}
\item PA$_2 (2,q,q)$, $q$ power of 2~\cite{Rao61}.
\item PA$_1 (2,q,q)$, $q$ odd prime power~\cite{Rao61}.
\item PA$_2 (2,10,10)$~\cite{BBE96}.
\end{enumerate}
\end{theorem}

For $k>2$ and small $n$, perpendicular arrays and ordered designs will be used in Section~\ref{sec:large} to
generate LS$(t,k,n;\mu)$ for small values of $n$. Related perpendicular arrays and ordered designs are given in the
following theorem.

\begin{theorem}
\label{thm:PA_nopairs}
The following perpendicular arrays and ordered designs exist:
\begin{enumerate}
\item PA$_3 (3,7,7)$~\cite{BBE96}.
\item OD$_1 (4,11,11)$~\cite{Bie06}.
\end{enumerate}
\end{theorem}

For more information on parameters of perpendicular arrays and ordered designs
the reader is referred to the short survey in~\cite{Bie06}.

\subsection{Configurations}
\label{sec:config}

Most of the constructions which will be described in the sequel
and those which were already mentioned (e.g. see Theorem~\ref{thm:LSH}),
are based on partitions of the point set.
Most of our partitions will be with parts of equal size, unless when there are only two parts,
where the two parts might not be of equal size.

Assume that the point set $Q$ has size $n$.
This set can be partitioned into two subsets
$A$ and~$B$, where $|A|= \ell$ and $|B|=n-\ell$. Given such a partition and a $k$-subset
$X = \{ x_1,x_2,\ldots,x_k \}$ of $Q$, we say that $X$ is a $k$-subset from \emph{configuration} $(i,j)$, where $i+j=k$, if
$|X \cap A|=i$ and $|X \cap B|=j$. The definition of configuration is generalized for a partition of the point set
into more than two parts. For example, if $Q$ is partitioned into four sets $A_1$, $A_2$, $A_3$, and $A_4$ (usually
of equal size), then
the $k$-subset $X$ is from configuration $(i_1,i_2,i_3,i_4)$, where $i_1+i_2+i_3+i_4=k$, if
$|X \cap A_j|=i_j$ for $1 \leq j \leq 4$.

The partition of the point set has a few advantages such as applying certain permutations
on one set of points, or considering each configuration separately makes it easier to
analyse the structure of the $k$-subsets.

\section{New Large Sets with Small Parameters}
\label{sec:large}

In this section we present some basic constructions for large sets with multiplicity and
for large sets of H-designs. The number of groups in these constructions of H-designs
is small and the same is true for the number of points in the large sets with multiplicity.
In Section~\ref{sec:small} large sets of H-designs with small
parameters are obtained by ad-hoc constructions.
In Section~\ref{sec:large_permute} large sets with multiplicity and small parameters are obtained
by applying sets of permutations on the coordinates of some
Steiner systems. These large sets with multiplicity imply related large sets of H-designs.

\subsection{Large Sets of H-designs with Small Number of Groups}
\label{sec:small}

This section is devoted to a few ad-hoc constructions presented in the proofs
of the next three lemmas.

\begin{lemma}
\label{lem:LH5}
There exists an LH$(5,4u,4,3)$ for any positive integer $u$.
\end{lemma}
\begin{proof}
We construct an H$(5,4,4,3)$ on the point set $\mathbb{Z}_{8}\cup\{8,9,\ldots,19\}$ with group set $\{\{0,2,4,6\},\{1,3,5,7\}\}\cup\{\{i,i+3,i+6,i+9\}:i=8,9,10\}$. Its block set $\B_0$ is as follows:
\begin{center}
{
\small
\begin{tabular}{lllllllllll}
$\{0,1,8,9\}$\ \ & $\{2,5,8,9\}$\ \ & $\{3,8,9,10\}$\ \ & $\{4,8,9,13\}$\ \ & $\{6,8,9,16\}$\ \ & $\{7,8,9,19\}$\\
$\{7,0,8,10\}$\ \ & $\{1,4,8,10\}$\ \ & $\{2,8,10,12\}$\ \ & $\{5,8,10,15\}$\ \ & $\{6,8,10,18\}$\ \ & $\{5,6,8,13\}$\\
$\{6,7,8,15\}$\ \ & $\{0,3,8,15\}$\ \ & $\{7,2,8,13$\}\ \ & $\{1,8,13,15\}$\ \ & $\{0,8,13,18\}$\ \ & $\{3,8,12,13\}$\\
$\{4,5,8,12\}$\ \ & $\{7,8,12,16\}$\ \ & $\{4,7,8,18\}$\ \ & $\{0,8,12,19\}$\ \ & $\{5,8,18,19\}$\ \ & $\{5,0,8,16\}$\\
$\{6,1,8,12\}$\ \ & $\{1,2,8,19\}$\ \ & $\{1,8,16,18\}$\ \ & $\{2,3,8,18\}$\ \ & $\{2,8,15,16\}$\ \ & $\{3,4,8,16\}$\\
$\{3,6,8,19\}$\ \ & $\{4,8,15,19\}$\ \ & $\{0,1,11,13\}$\ \ & $\{7,0,12,13\}$\ \ & $\{2,5,12,13\}$\ \ & $\{4,7,11,13\}$\\
$\{6,11,12,13\}$\ \ & $\{1,12,13,14\}$\ \ & $\{4,12,13,17\}$\ \ & $\{1,4,13,18\}$\ \ & $\{6,7,13,18\}$\ \ & $\{1,2,9,13\}$\\
$\{2,3,13,17\}$\ \ & $\{5,13,17,18\}$\ \ & $\{6,1,13,17\}$\ \ & $\{0,3,9,13\}$\ \ & $\{0,13,15,17\}$\ \ & $\{3,6,13,15\}$\\
$\{5,0,13,14\}$\ \ & $\{5,9,11,13\}$\ \ & $\{4,5,13,15\}$\ \ & $\{3,4,13,14\}$\ \ & $\{3,11,13,18\}$\ \ & $\{2,11,13,15\}$\\
$\{7,13,14,15\}$\ \ & $\{6,9,13,14\}$\ \ & $\{2,13,14,18\}$\ \ & $\{7,9,13,17\}$\ \ & $\{5,0,15,19\}$\ \ & $\{2,3,15,19\}$\\
$\{5,0,9,10\}$\ \ & $\{5,0,12,17\}$\ \ & $\{5,0,11,18\}$\ \ & $\{1,2,11,18\}$\ \ & $\{6,7,9,10\}$\ \ & $\{0,1,10,12\}$\\
$\{0,1,14,15\}$\ \ & $\{0,1,16,17\}$\ \ & $\{0,1,18,19\}$\ \ & $\{7,2\ 18,19\}$\ \ & $\{0,9,17,19\}$\ \ & $\{7,0,17,18\}$\\
$\{3,6,17,18\}$\ \ & $\{5,6,9,17\}$\ \ & $\{3,9,16,17\}$\ \ & $\{0,3,10,17\}$\ \ & $\{0,10,14,18\}$\ \ & $\{0,3,16,18\}$\\
$\{3,14,18,19\}$\ \ & $\{0,3,12,14\}$\ \ & $\{0,11,12,16\}$\ \ & $\{7,0,15\ 16\}$\ \ & $\{0,10,11,15\}$\ \ & $\{0,3,11,19\}$\\
$\{7,0,9,11\}$\ \ & $\{7,0,14,19\}$\ \ & $\{0,9,14,16\}$\ \ & $\{4,5,9,16\}$\ \ & $\{5,9,14,19\}$\ \ & $\{2,5,17,19\}$\\
$\{5,15,16,17\}$\ \ & $\{5,6,11,15\}$\ \ & $\{2,5,11,16\}$\ \ & $\{1,9,11,16\}$\ \ & $\{3,6,9,11\}$\ \ & $\{6,1,10,11\}$\\
$\{6,1,15,16\}$\ \ & $\{6,1,9,19\}$\ \ & $\{1,4,14,19\}$\ \ & $\{1,9,10,14\}$\ \ & $\{1,4,9,17\}$\ \ & $\{1,4,11,15\}$\\
$\{1,4,12,16\}$\ \ & $\{1,2,12,17\}$\ \ & $\{3,4,15,17\}$\ \ & $\{3,10,14,15\}$\ \ & $\{2,3,10,11\}$\ \ & $\{2,3,12,16\}$\\
$\{3,6,14,16\}$\ \ & $\{5,12,14,16\}$\ \ & $\{6,7,12,14\}$\ \ & $\{5,6,10,14\}$\ \ & $\{5,6,16,18\}$\ \ & $\{6,7,11,16\}$\\
$\{6,7,17,19\}$\ \ & $\{5,6,12,19\}$\ \ & $\{5,10,11,12\}$\ \ & $\{7,10,11,18\}$\ \ & $\{1,2,10,15\}$\ \ & $\{1,10,17,18\}$\\
$\{2,9,10,17\}$\ \ & $\{1,2,14,16\}$\ \ & $\{6,1,14,18\}$\ \ & $\{1,11,12,19\}$\ \ & $\{7,2,11,12\}$\ \ & $\{1,15,17,19\}$\\
$\{2,3,9,14\}$\ \ & $\{4,7,9,14\}$\ \ & $\{4,10,12,14\}$\ \ & $\{7,10,12,17\}$\ \ & $\{7,2,15,17\}$\ \ & $\{2,16,17,18\}$\\
$\{2,5,10,18\}$\ \ & $\{3,4,10,18\}$\ \ & $\{2,5,14,15\}$\ \ & $\{7,2,9,16\}$\ \ & $\{7,2,10,14\}$\ \ & $\{2,9,11,19\}$\\
$\{2,12,14,19\}$\ \ & $\{3,4,9,19\}$\ \ & $\{3,4,11,12\}$\ \ & $\{3,6,10,12\}$\ \ & $\{3,11,15,16\}$\ \ & $\{3,12,17,19\}$\\
$\{4,5,10,17\}$\ \ & $\{4,5,11,19\}$\ \ & $\{4,17,18,19\}$\ \ & $\{4,5,14,18\}$\ \ & $\{4,11,16,18\}$\ \ & $\{4,7,10,15\}$\\
$\{4,7,12,19\}$\ \ & $\{4,7,16,17\}$\ \ & $\{4,9,10,11\}$\ \ & $\{4,14,15,16\}$\ \ & $\{6,10,15,17\}$\ \ & $\{6,11,18,19\}$\\
$\{6,12,16,17\}$\ \ & $\{6,14,15,19\}$\ \ & $\{7,11,15,19\}$\ \ & $\{7,14,16,18\}$
\end{tabular}
}
\end{center}
We apply to $\B_0$ the automorphism group $(\mathbb{Z}_8,+)$ (only to the points of $\mathbb{Z}_8$,
leaving each one of the points in the set $\{8,9,\ldots,19\}$ unchanged)
to obtain 8 mutually disjoint H$(5,4,4,3)$s, namely an LH$(5,4,4,3)$.
Finally, an LH$(5,4u,4,3)$ for any $u \geq 1$ is obtained by Theorem~\ref{thm:expand}.
\end{proof}

\begin{lemma}
\label{lem:LH6}
There exists an LH$(6,g,4,3)$ if and only if $g$ is divisible by $3$.
\end{lemma}
\begin{proof}
First we construct an H$(6,3,4,3)$ on the point set
$\mathbb{Z}_{9}\cup\{9,10,\ldots,17\}$ with group set $\{\{i,i+3,i+6\}:i=0,1,2,9,10,11\}$. Its block set $\cB_0$ is as follows:
\begin{center}
{
\small
\begin{tabular}{lllllllllll}
$\{0,1,2,9\}$ & $\{8,0,1,10\}$ & $\{0,1,5,11\}$ & $\{0,1,12,14\}$ & $\{0,1,13,15\}$ & $\{0,1,16,17\}$\\
$\{3,5,12,14\}$ & $\{4,8,12,14\}$ & $\{7,10,12,14\}$ & $\{2,12,13,14\}$ & $\{6,12,14,16\}$ & $\{0,2,11,13\}$\\
$\{1,2,13,17\}$ & $\{2,3,4,13\}$ & $\{2,6,9,13\}$ & $\{7,2,13,15\}$ &$\{1,9,11,13\}$ & $\{3,9,13,14\}$\\
$\{4,5,9,13\}$ & $\{7,0,9,13\}$ & $\{8,9,13,17\}$ & $\{5,0,13,17\}$ & $\{4,6,13,17\}$ & $\{3,13,15,17\}$\\
$\{6,8,13,15\}$ & $\{7,12,13,17\}$ & $\{0,4,9,17\}$ & $\{1,9,10,17\}$ & $\{0,4,12,13\}$ & $\{8,0,13,14\}$\\
$\{5,7,13,14\}$ & $\{1,3,5,13\}$  &$\{8,1,12,13\}$ & $\{6,1,13,14\}$ & $\{3,11,12,13\}$ & $\{7,8,3,13\}$\\
$\{6,7,11,13\}$ & $\{4,8,11,13\}$ & $\{4,13,14,15\}$ & $\{5,6,12,13\}$ &$\{5,11,13,15\}$ & $\{8,0,4,15\}$\\
$\{0,4,11,16\}$ & $\{4,5,0,10\}$ & $\{0,2,4,14\}$ & $\{5,0,9,14\}$ & $\{0,10,14,15\}$ & $\{7,0,11,15\}$\\
$\{0,2,15,17\}$ & $\{5,0,15,16\}$ & $\{5,7,0,12\}$ & $\{7,0,14,16\}$ & $\{7,8,0,17\}$ & $\{0,10,12,17\}$\\
$\{7,0,2,10\}$ & $\{0,2,12,16\}$ & $\{8,0,11,12\}$ & $\{0,9,10,11\}$ & $\{8,0,9,16\}$ & $\{1,2,6,16\}$\\
$\{1,2,10,12\}$ & $\{6,7,8,12\}$ & $\{2,3,7,12\}$ & $\{6,7,2,14\}$  & $\{7,2,16,17\}$ & $\{7,2,9,11\}$\\
$\{6,8,9,11\}$ & $\{6,8,10,14\}$ & $\{4,6,9,14\}$ & $\{6,9,16,17\}$ & $\{5,7,9,17\}$ & $\{8,1,3,9\}$\\
$\{8,1,11,16\}$ & $\{8,1,14,15\}$ & $\{5,6,14,15\}$ & $\{5,6,7,16\}$ & $\{6,11,15,16\}$ & $\{8,3,11,15\}$\\
$\{4,10,11,15\}$ & $\{4,6,10,12\}$ & $\{2,4,11,12\}$ & $\{4,5,6,11\}$ & $\{3,4,5,15\}$ & $\{5,7,10,15\}$\\
$\{3,5,7,11\}$ & $\{5,10,11,12\}$ & $\{5,9,11,16\}$ & $\{2,4,9,16\}$ & $\{2,4,10,17\}$ & $\{2,6,12,17\}$\\
$\{6,1,11,12\}$ & $\{1,3,10,11\}$ & $\{7,8,10,11\}$ & $\{4,8,9,10\}$ & $\{3,4,9,11\}$ & $\{1,2,3,14\}$\\
$\{2,9,10,14\}$ & $\{3,5,9,10\}$ & $\{6,7,9,10\}$ & $\{6,7,15,17\}$ & $\{1,5,15,17\}$ & $\{6,8,1,17\}$\\
$\{1,2,11,15\}$ & $\{1,3,12,17\}$ & $\{3,5,16,17\}$ & $\{4,5,12,17\}$ & $\{4,5,14,16\}$ & $\{1,3,15,16\}$\\
$\{2,14,15,16\}$ & $\{5,6,1,9\}$ & $\{1,5,10,14\}$ & $\{1,5,12,16\}$ & $\{8,3,10,12\}$ & $\{3,4,8,17\}$\\
$\{6,1,10,15\}$ & $\{1,9,14,16\}$ & $\{2,3,9,17\}$ & $\{2,3,10,15\}$ & $\{2,3,11,16\}$ & $\{2,4,6,15\}$\\
$\{2,6,10,11\}$ & $\{3,4,10,14\}$ & $\{3,4,12,16\}$ & $\{3,7,9,16\}$ & $\{3,7,10,17\}$ & $\{3,7,14,15\}$\\
$\{8,3,14,16\}$ & $\{4,6,8,16\}$& $\{4,15,16,17\}$ & $\{5,6,10,17\}$ & $\{7,8,9,14\}$ & $\{7,8,15,16\}$\\
$\{7,11,12,16\}$ & $\{8,10,15,1\}$ & $\{8,12,16,17\}$

 \end{tabular}
}
\end{center}
We apply to ${\cal B}_0$ the automorphism group $(\mathbb{Z}_9,+)$ (only to the points of $\mathbb{Z}_9$, leaving each one
of the points in the set $\{9,10,\ldots,17\}$ unchanged)
to obtain 9 mutually disjoint H$(6,3,4,3)$s, namely an  LH$(6,3,4,3)$.
Finally, an LH$(6,g,4,3)$ for any $g$ divisible by 3 is obtained by Theorem~\ref{thm:expand}.
For $g$ which is not divisible by 3, an LH$(6,g,4,3)$
does not exist by Theorem~\ref{thm:necess}.
\end{proof}

\begin{lemma}
\label{lem:LH7}
There exists an LH$(7,g,4,3)$ if and only if $g$ is even.
\end{lemma}
\begin{proof}
First we construct an H$(7,2,4,3)$ on the point set $\Z_{14}$ with groups $\{i,i+7\},~ 0\le i \le 6$.
Its block set $\cB_0$ is as follows:
\begin{center}
{
\small
\begin{tabular}{lllllllllll}
$\{0,1,2,3\}$ & $\{0,1,4,5\}$ & $\{0,1,6,9\}$ & $\{0,1,10,11\}$ & $\{0,1,12,13\}$ & $\{0,2,4,6\}$ & $\{0,2,8,12\}$ \\
$\{0,2,10,13\}$ & $\{0,2,5,11\}$ & $\{0,4,10,12\}$ & $\{0,6,8,10\}$ & $\{0,3,5,6\}$ & $\{0,6,11,12\}$ & $\{0,3,9,12\}$\\
$\{0,5,9,10\}$ & $\{0,5,8,13\}$ & $\{0,3,4,13\}$ & $\{0,3,8,11\}$ & $\{0,9,11,13\}$ & $\{0,4,8,9\}$ & $\{1,2,4,13\}$ \\
$\{2,3,5,13\}$ & $\{2,8,11,13\}$ & $\{2,7,12,13\}$ & $\{2,3,4,12\}$ & $\{2,3,7,8\}$ & $\{2,3,6,11\}$ & $\{2,5,6,8\}$ \\
$\{2,4,8,10\}$ & $\{2,4,5,7\}$ & $\{1,2,5,10\}$ & $\{3,4,5,8\}$ & $\{1,2,6,7\}$ & $\{1,2,11,12\}$ & $\{2,7,10,11\}$ \\
$\{1,5,6,11\}$ & $\{6,9,10,11\}$ & $\{6,7,8,11\}$ & $\{4,6,8,12\}$ & $\{4,7,8,13\}$ & $\{2,6,10,12\}$ & $\{1,4,6,10\}$\\
$\{1,3,6,12\}$ & $\{3,7,11,12\}$ & $\{3,5,9,11\}$ & $\{1,3,5,7\}$ & $\{3,7,9,13\}$ & $\{3,4,6,7\}$ & $\{1,3,4,9\}$ \\
$\{1,4,7,12\}$ & $\{1,9,10,12\}$ & $\{1,5,9,13\}$ & $\{4,5,6,9\}$ & $\{5,7,8,9\}$ & $\{5,8,10,11\}$ & $\{5,7,11,13\}$ \\
$\{4,5,10,13\}$ & $\{8,9,10,13\}$ & $\{4,9,12,13\}$ & $\{1,3,11,13\}$ & $\{1,7,9,11\}$ & $\{1,7,10,13\}$ & $\{3,6,8,9\}$ \\
$\{3,8,12,13\}$ & $\{4,7,9,10\}$ & $\{5,6,7,10\}$ & $\{6,7,9,12\}$ & $\{7,8,10,12\}$ & $\{8,9,11,12\}$ & $\{10,11,12,13\}$
\end{tabular}
}
\end{center}
We apply 8 permutations of the following $8\times 14$ array on these blocks to ${\cal B}_0$ to generate an LH$(7,2,4,3)$.
\begin{center}
{
\begin{tabular}{lllllllllll}

0\ 1\ 2\ 3\ 4\ 5\ 6\ 7\ 8\ 9\ 10\ 11\ 12\ 13\\

0\ 8\ 6\ 9\ 3\ 12\ 4\ 7\ 1\ 13\ 2\ 10\ 5\ 11\\

0\ 12\ 11\ 8\ 13\ 3\ 2\ 7\ 5\ 4\ 1\ 6\ 10\ 9\\

1\ 3\ 9\ 7\ 4\ 6\ 12\ 8\ 10\ 2\ 0\ 11\ 13\ 5\\

1\ 4\ 2\ 5\ 10\ 13\ 7\ 8\ 11\ 9\ 12\ 3\ 6\ 0\\

1\ 12\ 4\ 0\ 10\ 13\ 9\ 8\ 5\ 11\ 7\ 3\ 6\ 2\\

2\ 0\ 4\ 5\ 3\ 6\ 1\ 9\ 7\ 11\ 12\ 10\ 13\ 8\\

2\ 3\ 5\ 1\ 13\ 11\ 7\ 9\ 10\ 12\ 8\ 6\ 4\ 0

\end{tabular}
}
\end{center}
Finally, an LH$(7,g,4,3)$ for any even $g$ is obtained by applying Theorem~\ref{thm:expand}
on this large set. For odd $g$, an LH$(7,g,4,3)$
does not exist by Theorem~\ref{thm:necess} which completes the claim of the lemma.
\end{proof}

The conclusion of Lemma~\ref{lem:LH5}, Lemma~\ref{lem:LH6}, Lemma~\ref{lem:LH7},
and Theorem~\ref{thm:necess} is the following result.

\begin{theorem}
\label{thm:LH789}
A large set of H-designs LH$(n,g,4,3)$ exists for all admissible parameters
with $n \in \{ 5,6,7\}$, with possible exceptions for $n=5$ and $g \equiv 2~ (\text{mod}~4)$.
\end{theorem}

\subsection{Large Sets of H-designs from Large Sets with Multiplicity}
\label{sec:large_permute}

The goal of this section is to obtain large sets with multiplicity and then
to obtain large sets of H-designs using Theorem~\ref{thm:LSH}.
The first step in this direction is to find a large set LS$(t,k,n;g^{k-t})$, for different
sets of parameters, where the multiplicity $g^{k-t}$ is not too large.
A special case of Theorem~\ref{thm:LSH}, which is the key to obtain large set of H-designs from
large set with multiplicity, was applied in~\cite{Etz96}
to obtain large set LH$(2^m,g,4,3)$, $m \geq 3$, for each $g \geq 2$ (see Theorem~\ref{thm:LHpower 2}).
If a construction of an
LS$(t,k,n)$ is known then clearly there exists an LS$(t,k,n;g^{k-t})$ for any $g \geq 1$
and Theorem~\ref{thm:LSH} can be applied trivially.
Hence, our first target is to construct such large sets LS$(3,4,n;\mu)$, where $n$ is not a power of 2 and
$\mu$ small as possible. Our second target is to construct large sets
LS$(t,k,n;g^{k-t})$ for $k > t >3$.

Let $S$ be Steiner system S$(t,k,n)$ on the point set $\Z_n$. We would like to use this Steiner system to form an
LS$(t,k,n;\mu)$ on the point set~$\Z_n$.
The main idea is to use a set of permutations on~$\Z_n$, to form isomorphic
systems to $S$, such that the union of the permutations
restricted to the $k$-subsets of $S$ yield each $k$-subset of $\Z_n$ exactly $\mu$ times.
This is the point that we want to have OD$_\lambda (k,n,n)$ or PA$_\lambda (k,n,n)$
with $\lambda$ small as possible.
The simplest set of such permutations are all the $n!$ permutations of $S_n$.
The outcome is $n!$ Steiner systems S$(t,k,n)$
such that each $k$-subset of $\Z_n$ is contained in exactly
$\frac{\binom{n}{t}}{\binom{k}{t}} n! {\Huge / } \binom{n}{k} =\frac{n!(n-k)!(k-t)!}{(n-t)!}$ of these systems.
This is a simple way to obtain a large set LS$(t,k,n;\mu)$, where $\mu=\frac{n!(n-k)!(k-t)!}{(n-t)!}$. The
multiplicity $\mu$ of this large set is very large and our target is to obtain such a large set with a much smaller multiplicity $\mu$.
The multiplicity can be cut to half if we use only the $n!/2$ even pemutations instead of all the $n!$ permutations of $S_n$.
As pointed in Section~\ref{sec:perm_arrays}, there are no known such array, for $k>3$, with a reasonable $\lambda$.
This is also true about the probabilistic argument for the existence of such arrays in~\cite{KLP17} and for
the probabilistic construction presented in~\cite{FPY15}.
Generally, to find a smaller set of permutations, for this purpose, is an interesting open problem for itself.
Theorem~\ref{thm:perm_LS} concludes this summary by presenting the parameters, when
the permutations of a perpendicular array PA$_\lambda (k,n,n)$ are applied on all the blocks of a Steiner system S$(t,k,n)$.
For small parameters we can use at least three different strategies. The first one is to construct
large sets with multiplicity using computer search. The second one is by applying ordered designs
or perpendicular arrays with small parameters. The third one is by ad-hoc constructions.
A few examples for all these methods are given in this subsection.

\begin{example}
\label{ex:S11}
We have used a computer search to find a large set LS$(4,5,11;2)$, on the point set $\Z_{11}$, with 14 Steiner
systems S$(4,5,11)$. The first system $S_1$ has the following 66 blocks.
\begin{center}
{
\begin{tabular}{lllllllllll}
$\{0, 1, 2, 3, 4\}$\ \ & $\{0, 1, 2, 5, 6\}$\ \ & $\{0, 1, 2, 7, 8\}$\ \ & $\{0, 1, 2, 9, 10\}$\ \ & $\{0, 1, 3, 5, 7\}$\ \ & $\{0,1,3,6,9\}$\\
$\{0, 1, 3, 8, 10\}$\ \ & $\{0, 1, 4, 5, 10\}$\ \ & $\{0, 1, 5, 8, 9\}$\ \ & $\{0, 1, 4, 6, 8\}$\ \ & $\{0, 1, 4, 7, 9\}$\ \ & $\{0,1,6,7,10\}$\\
$\{0, 2, 3, 5, 8\}$\ \ & $\{0, 2, 3, 7, 9\}$\ \ & $\{0, 2, 3, 6, 10\}$\ \ & $\{0, 2, 4, 8, 10\}$\ \ & $\{0, 2, 6, 8, 9\}$\ \ & $\{0,2,4,6,7\}$\\
$\{0, 2, 5, 7, 10\}$\ \ & $\{0, 5, 6, 8, 10\}$\ \ & $\{0, 4, 5, 7, 8\}$\ \ & $\{0, 3, 4, 7, 10\}$\ \ & $\{0, 3, 5, 9, 10\}$\ \ & $\{0,3,4,5,6\}$\\
$\{0, 2, 4, 5, 9\}$\ \ & $\{0, 3, 4, 8, 9\}$\ \ & $\{0, 3, 6, 7, 8\}$\ \ & $\{0, 4, 6, 9, 10\}$\ \ & $\{0, 5, 6, 7, 9\}$\ \ & $\{0,7,8,9,10\}$\\
$\{1, 2, 3, 5, 9\}$\ \ & $\{1, 2, 3, 7, 10\}$\ \ & $\{1, 2, 6, 7, 9\}$\ \ & $\{1, 4, 5, 6, 9\}$\ \ & $\{1, 2, 4, 5, 7\}$\ \ & $\{1,2,4,8,9\}$\\
$\{1, 4, 7, 8, 10\}$\ \ & $\{2, 3, 4, 7, 8\}$\ \ & $\{1, 2, 3, 6, 8\}$\ \ & $\{1, 2, 4, 6, 10\}$\ \ & $\{3, 4, 6, 8, 10\}$\ \ & $\{3,5,7,8,10\}$\\
$\{3, 4, 5, 7, 9\}$\ \ & $\{3, 6, 7, 9, 10\}$\ \ & $\{1, 2, 5, 8, 10\}$\ \ & $\{2, 5, 6, 9, 10\}$\ \ & $\{1, 3, 4, 5, 8\}$\ \ & $\{1,3,4,6,7\}$\\
$\{1, 3, 4, 9, 10\}$\ \ & $\{2, 3, 4, 6, 9\}$\ \ & $\{1, 3, 5, 6, 10\}$\ \ & $\{1, 3, 7, 8, 9\}$\ \ & $\{1, 5, 6, 7, 8\}$\ \ & $\{2,3,5,6,7\}$\\
$\{1, 5, 7, 9, 10\}$\ \ & $\{1, 6, 8, 9, 10\}$\ \ & $\{2, 3, 4, 5, 10\}$\ \ & $\{2, 3, 8, 9, 10\}$\ \ & $\{2, 4, 5, 6, 8\}$\ \ & $\{2,4,7,9,10\}$\\
$\{2, 5, 7, 8, 9\}$\ \ & $\{2, 6, 7, 8, 10\}$\ \ & $\{3, 5, 6, 8, 9\}$\ \ & $\{4, 5, 6, 7, 10\}$\ \ & $\{4, 5, 8, 9, 10\}$\ \ & $\{4,6,7,8,9\}$
\end{tabular}
}
\end{center}
This system and the 13 systems obtained by applying the following 13 coordinate permutations
\begin{center}
{
\begin{tabular}{lllllllllllllll}

$0$\ $1$\ $2$\ $3$\ $4$\ $6$\ $5$\ $9$\ $10$\ $7$\ $8$\\

$0$\ $1$\ $2$\ $3$\ $6$\ $7$\ $9$\ $4$\ $5$\ $10$\ $8$\\

$0$\ $1$\ $2$\ $3$\ $6$\ $8$\ $9$\ $7$\ $10$\ $4$\ $5$\\

$0$\ $1$\ $2$\ $3$\ $8$\ $5$\ $10$\ $9$\ $6$\ $4$\ $7$\\

$0$\ $1$\ $2$\ $3$\ $9$\ $5$\ $8$\ $4$\ $10$\ $6$\ $7$\\

$0$\ $1$\ $2$\ $4$\ $9$\ $5$\ $7$\ $8$\ $3$\ $6$\ $10$\\

$0$\ $1$\ $2$\ $5$\ $10$\ $4$\ $8$\ $7$\ $6$\ $3$\ $9$\\

$0$\ $1$\ $2$\ $6$\ $9$\ $7$\ $4$\ $8$\ $5$\ $10$\ $3$\\

$0$\ $1$\ $2$\ $7$\ $5$\ $4$\ $9$\ $10$\ $3$\ $8$\ $6$\\

$0$\ $1$\ $2$\ $7$\ $3$\ $10$\ $6$\ $8$\ $4$\ $9$\ $5$\\

$0$\ $1$\ $2$\ $8$\ $6$\ $10$\ $4$\ $5$\ $9$\ $7$\ $3$\\

$0$\ $1$\ $2$\ $8$\ $10$\ $6$\ $4$\ $9$\ $7$\ $5$\ $3$\\

$0$\ $1$\ $2$\ $9$\ $8$\ $4$\ $6$\ $5$\ $3$\ $7$\ $10$
 \end{tabular}
}
\end{center}
yield the large set LS$(4,5,11;2)$.
\end{example}

Example~\ref{ex:S11} yields the following result.

\begin{lemma}
\label{lem:3LS}
There exist the following three large sets with multiplicity:
\begin{center}
LS$(4,5,11;2)$, LS$(3,4,10;2)$,  LS$(5,6,12;2)$.
\end{center}
\end{lemma}
\begin{proof}
$~$
\begin{enumerate}
\item
The large set LS$(4,5,11;2)$ was constructed in Example~\ref{ex:S11} using $S_1$,
the Steiner system S$(4,5,11)$, and the given 13 permutations on $\Z_{11}$.

\item
Each system of the large set LS$(4,5,11;2)$ has exactly 30 blocks containing the point 10.
By considering these 30 blocks and removing the point 10
from each one of them yields an LS$(3,4,10;2)$. This LS$(3,4,10;2)$ consists
of 14 Steiner systems S$(3,4,10)$, on the point set $\Z_{10}$, which are the derived systems of the 14 Steiner
systems S$(4,5,11)$ of the large set LS$(4,5,11;2)$.

\item
Let $S_1,S_2,\ldots,S_{14}$ be the 14 systems of the LS$(4,5,11;2)$.
Let $S'_i \triangleq \{ X \cup \{ 11 \} ~:~  X \in S_i \} \cup \{ \Z_{11} \setminus X ~:~ X \in S_i \}$, for each $1 \leq i \leq 14$.
It is well-known that each such $S'_i$ is a Steiner system S$(5,6,12)$. It implies that
$\{ S'_i ~:~ 1 \leq i \leq 14 \}$ is a large set LS$(5,6,12;2)$.
\end{enumerate}
\end{proof}

\begin{example}
\label{ex:S10}

Let $S$ be the Steiner system S$(3,4,10)$ obtained as the derived system of the Steiner system S$(4,5,11)$
presented in Example~\ref{ex:S11}, i.e $\{ X \setminus \{11\} ~:~ 11 \in X,~ X \in S_1 \}$.
This system and the 20 systems obtained by applying the following 20 coordinates permutations
\begin{center}
{

  \begin{tabular}{lllllllllll}

$0$\ $6$\ $2$\ $5$\ $4$\ $9$\ $7$\ $1$\ $8$\ $3$\\

$0$\ $5$\ $2$\ $1$\ $8$\ $6$\ $9$\ $3$\ $4$\ $7$\\

$0$\ $4$\ $2$\ $7$\ $8$\ $1$\ $9$\ $5$\ $6$\ $3$\\

$0$\ $9$\ $2$\ $5$\ $8$\ $3$\ $1$\ $4$\ $6$\ $7$\\

$0$\ $9$\ $2$\ $7$\ $8$\ $4$\ $1$\ $6$\ $3$\ $5$\\

$0$\ $1$\ $2$\ $7$\ $8$\ $5$\ $4$\ $3$\ $9$\ $6$\\

$0$\ $6$\ $2$\ $4$\ $8$\ $7$\ $9$\ $3$\ $5$\ $1$\\

$0$\ $9$\ $2$\ $5$\ $6$\ $8$\ $4$\ $7$\ $1$\ $3$\\

$0$\ $7$\ $2$\ $5$\ $3$\ $6$\ $9$\ $4$\ $8$\ $1$\\

$0$\ $6$\ $2$\ $9$\ $3$\ $8$\ $4$\ $1$\ $7$\ $5$\\

$0$\ $6$\ $2$\ $7$\ $9$\ $1$\ $4$\ $5$\ $3$\ $8$\\

$0$\ $5$\ $2$\ $4$\ $9$\ $8$\ $6$\ $1$\ $7$\ $3$\\

$0$\ $8$\ $2$\ $4$\ $1$\ $9$\ $3$\ $6$\ $7$\ $5$\\

$0$\ $4$\ $2$\ $5$\ $7$\ $1$\ $6$\ $8$\ $3$\ $9$\\

$0$\ $5$\ $2$\ $8$\ $7$\ $3$\ $9$\ $1$\ $4$\ $6$\\

$0$\ $6$\ $2$\ $8$\ $5$\ $4$\ $3$\ $1$\ $9$\ $7$\\

$0$\ $5$\ $4$\ $7$\ $2$\ $3$\ $9$\ $6$\ $1$\ $8$\\

$0$\ $6$\ $4$\ $5$\ $8$\ $1$\ $7$\ $2$\ $9$\ $3$\\

$0$\ $1$\ $4$\ $9$\ $6$\ $5$\ $3$\ $2$\ $8$\ $7$\\

$0$\ $2$\ $4$\ $8$\ $6$\ $9$\ $3$\ $1$\ $7$\ $5$   \end{tabular}
}
\end{center}
yield a large set LS$(3,4,10;3)$.

\end{example}

The large set LS$(3,4,10;3)$ of Example~\ref{ex:S10} and the large set LS$(3,4,10;2)$ obtained in Lemma~\ref{lem:3LS} imply
the following theorem.

\begin{theorem}
\label{thm:LS10}
For each $\mu \geq 2$ there exists a large set LS$(3,4,10;\mu)$.
\end{theorem}
Applying Theorem~\ref{thm:LSH} on the large set LS$(3,4,10;\mu)$ implies the following theorem.
\begin{theorem}
\label{thm:LH10}
For each $g \geq 2$ there exists an LH$(10,g,4,3)$.
\end{theorem}

Theorem~\ref{thm:LH789}, Theorem~\ref{thm:LHpower 2}, Theorem~\ref{thm:LH10}, and Theorem~\ref{thm:necess} imply
the following conclusion.

\begin{corollary}
\label{cor:LH4_10}
A large set of H-designs LH$(n,g,4,3)$ exists for all admissible parameters
with $n \in \{ 4,5,6,7,8,10 \}$, with possible exceptions for $n=5$ and $g \equiv 2~ (\text{mod}~4)$.
\end{corollary}

\begin{theorem}
\label{thm:permLarge}
Let $S$ be an S$(3,4,n)$ on the point set $\Z_n$, where $n \geq 14$, and let $A$, $B$ be a partition of $\Z_n$ into
subsets of size $3$ and $n-3$, respectively.
If there exists a
perpendicular array PA$_\lambda (4,n-3,n-3)$, then there exists an LS$(3,4,n;\mu)$, where $\mu = \lambda \binom{n-4}{3} {\huge /} 4$.
\end{theorem}
\begin{proof}
First, we compute the number of blocks of $S$ in each configuration.
There is exactly one block in configuration $(3,1)$,
which implies that the unique subset from configuration $(3,0)$ is contained in one block of $S$.
There are exactly $\frac{3(n-3)-3}{2} = \frac{3(n-4)}{2}$ blocks in configuration $(2,2)$,
which implies that each subset from configuration $(2,0)$ is contained in $\gamma_2 =\frac{n-4}{2}$ blocks of $S$.
There are exactly $\left( 3 \binom{n-3}{2} - \frac{3(n-4)}{2} 2  \right) {\huge /} 3 = \frac{(n-4)(n-5)}{2}$ blocks in configuration $(1,3)$,
which implies that each subset from configuration $(1,0)$ is contained in $\gamma_1 =\frac{(n-4)(n-5)}{6}$ blocks of $S$.
There are exactly $\gamma_0 =\left( \binom{n-3}{3} - \frac{(n-4)(n-5)}{2}  \right) {\huge /} 4 = \binom{n-4}{3} {\huge /} 4$
blocks in configuration $(4,0)$.

Assume $M$ is a perpendicular array PA$_\lambda (4,n-3,n-3)$.
Let $S'$ be the set of blocks obtained by applying the permutations of $M$ on the B-part of $S$.
The computations done implies that each subset of configuration $(3,1)$ is contained in $\lambda \binom{n-4}{3} {\huge /} 4$ subsets of $S'$.
Each subset of configuration $(2,2)$ is contained in
$\gamma_2 \lambda \binom{n-3}{4} {\huge /} \binom{n-3}{2}=\lambda \binom{n-4}{3} {\huge /} 4$ subsets of $S'$.
Each subset of configuration $(1,3)$ is contained in
$\gamma_1 \lambda \binom{n-3}{4} {\huge /} \binom{n-3}{3}=\lambda \binom{n-4}{3} {\huge /} 4$ subsets of $S'$.
Finally, each subset of configuration $(0,4)$ is contained in
$\gamma_0 \lambda \binom{n-3}{4} {\huge /} \binom{n-3}{4}=\lambda \binom{n-4}{3} {\huge /} 4$ subsets of $S'$.

Thus, $S'$ is an LS$(3,4,n;\mu)$, where $\mu = \lambda \binom{n-4}{3} {\huge /} 4$.
\end{proof}
\begin{corollary}
\label{cor:permLarge}
Let $S$ be an S$(3,4,n)$ on the point set $\Z_n$, where $n \geq 14$, and let $A$, $B$ be a partition of $\Z_n$ into
subsets of size $3$ and $n-3$, respectively.
If there exists an
ordered design OD$_\lambda (4,n-3,n-3)$, then there exists an LS$(3,4,n;\mu)$, where $\mu = 6 \lambda \binom{n-4}{3}$.
\end{corollary}

Theorem~\ref{thm:permLarge} and Corollary~\ref{cor:permLarge} can be generalized for other Steiner
systems S$(t,t+1,n)$. We omit these generalizations as there are no perpendicular arrays or ordered design
with a relatively small number of permutations.

Theorem~\ref{thm:permLarge} and Corollary~\ref{cor:permLarge} can also be modified and generalied to other
Steiner system S$(t,t+1,n)$, where the point set is partitioned into two subsets of size $t$ and $n-t$,
and also to the case when the point set is partitioned into two subsets of size $\ell$ and $n-\ell$,
where $\ell <t$. We omit the related discussion and theorems, since we don't have any example where a large set with
relatively small multiplicity is obtained. Contrary to these Corollary~\ref{cor:permLarge} can be applied
using OD$_1(4,11,11)$ (see Theorem~\ref{thm:PA_nopairs}) to obtain the following result.

\begin{corollary}
$~$
There exists an LS$(3,4,14;720)$.
\end{corollary}
\begin{corollary}
$~$
There exists an LH$(14,720,4,3)$.
\end{corollary}

The coordinate partitioning method which was mentioned above can be applied to Steiner system with specific
parameters and related perpendicular array. We demonstrate the idea for $S$,
a Steiner system S$(5,6,12)$. Assume that the system is
constructed on a point set partitioned into two subsets $A_1$ and $A_2$ of size 5 and 7, respectively.
First, we compute the number of blocks in $S$ for each configuration.
Configuration $(5,1)$ has exactly one block,
which implies that unique subset from configuration $(5,0)$ is contained in one block of $S$.
Configuration $(4,2)$ has 15 blocks,
which implies that each subset from configuration $(4,0)$ is contained in exactly 3 blocks of $S$.
Configuration $(3,3)$ has 50 blocks,
which implies that each subset from configuration $(3,0)$ is contained in exactly 5 blocks of $S$.
Configuration $(2,4)$ has 50 blocks,
which implies that each subset from configuration $(2,0)$ is contained in exactly 5 blocks of $S$.
Configuration $(1,5)$ has 15 blocks,
which implies that each subset from configuration $(1,0)$ is contained in exactly 3 blocks of $S$.
Finally, configuration $(5,1)$ has exactly one block.
Let $M$ be a perpendicular array $PA_3(3,7,7)$ (see Theorem~\ref{thm:PA_nopairs}).
Let $S'$ be the set of blocks obtained by applying the permutations of $M$ on the part $A_2$ of $S$.
The computations done imply that each subset of configuration $(5,1)$ is contained in $3 \binom{7}{3} {\huge /} 7=15$ subsets of $S'$.
Each subset of configuration $(4,2)$ is contained in
$15 \cdot 3 \binom{7}{3} {\huge /} \left( \binom{5}{4}\binom{7}{2} \right)= 15$ subsets of $S'$.
Each subset of configuration $(3,3)$ is contained in
$50 \cdot 3 \binom{7}{3} {\huge /} \left( \binom{5}{3}\binom{7}{3} \right)= 15$ subsets of $S'$.
Note, that the computations for configurations $(2,4)$, $(1,5)$, $(0,6)$ are the same
as for configurations $(3,3)$, $(4,2)$, $(5,1)$, respectively.
Thus, $S'$ is an LS$(5,6,12;15)$. Combining this result and Lemma~\ref{lem:3LS} we infer
the following theorem.

\begin{theorem}
There exists an LS$(4,5,11;\mu)$ and an LS$(5,6,12;\mu)$ for each $\mu \geq 2$,
with possible exceptions when $\mu \in \{ 3,5,7,9,11,13 \}$.
\end{theorem}

\begin{theorem}
There exists an LH$(11,g,5,4)$ and an LH$(12,g,6,5)$ for each $g \geq 2$,
with possible exceptions when $g \in \{ 3,5,7,9,11,13 \}$.
\end{theorem}

\section{The Main Ingredient Constructions}
\label{sec:review}

In this section we will discuss the main recursive constructions for
Steiner quadruple systems or more precisely, for a set of pairwise disjoint Steiner quadruple
systems. The first construction is a folklore doubling construction
which will be called the DB (for doubling) Construction.
The second one is also a doubling construction due to Lindner~\cite{Lin77}. It will
be called the DLS (for Doubling Lindner Systems) Construction. We will present a slightly different variant than the one
given in~\cite{Lin77}. This variant was already presented in~\cite{Etz96}.
The third construction is a quadrupling construction. It is a variant of the construction
presented in another paper of Lindner~\cite{Lin85}. It will be called the QLS (for Quadrupling Lindner systems) Construction.
While the two doubling constructions will be defined and discussed in details, for the quadrupling construction,
only some properties will be discussed, while the exact definition of the construction for our setup will
be given in Section~\ref{sec:five}.
These constructions will be used later in a variant for a construction of Etzion and Hartman~\cite{EtHa91}
which is a more complicated construction, in which instead of doubling, or quadrupling, the multiplication is
by $2^m$. As this will be the main construction and it is more complicated,
it will be presented separately in Section~\ref{sec:mainC}.

\subsection{DLS Construction (Doubling)}
\label{sec:DLS}

Let $(\Z_n,B)$ be an SQS$(n)$ and let $A$ be a $n \times n$ Latin square,
on the point set $\Z_n$, with
no $2 \times 2$ subsquares.
Denote by~$\alpha_i$ the permutation on $\Z_n$ defined by $\alpha_i(j)=y$ if and only if $A(i,j)=y$.
For each $i$, $0 \leq i \leq n-1$, we define a set of quadruples $B_i$ on $\Z_n \times \Z_2$ as follows:

\begin{enumerate}
\item For each quadruple $\{ x_1,x_2,x_3,x_4 \} \in B$, the following 8 quadruples are contained in~$B_i$.
$$
\{ (x_1,0),(x_2,0),(x_3,0),(\alpha_i(x_4),1) \},~~ \{ (x_1,1),(x_2,1),(x_3,1),(\alpha_i(x_4),0) \}
$$
$$
\{ (x_1,0),(x_2,0),(\alpha_i(x_3),1),(x_4,0) \},~~ \{ (x_1,1),(x_2,1),(\alpha_i(x_3),0),(x_4,1) \}
$$
$$
\{ (x_1,0),(\alpha_i(x_2),1),(x_3,0),(x_4,0) \},~~ \{ (x_1,1),(\alpha_i(x_2),0),(x_3,1),(x_4,1) \}
$$
$$
\{ (\alpha_i(x_1),1),(x_2,0),(x_3,0),(x_4,0) \},~~ \{ (\alpha_i(x_1),0),(x_2,1),(x_3,1),(x_4,1) \}
$$

\item For each pair $\{ x_1,x_2 \} \subset \Z_n$, the quadruple
$\{ (x_1,0),(x_2,0),(\alpha_i(x_1),1),(\alpha_i(x_2),1) \}$ is contained in~$B_i$.
\end{enumerate}

This DLS Construction is a variant~\cite{Etz96} of the Lindner Construction~\cite{Lin77}.
Each $B_i$ constructed via the DLS Construction in an SQS$(2n)$ and the set
$\{ B_i ~:~ 0 \leq i \leq n-1\}$ is a set of $n$ pairwise disjoint SQS$(2n)$.

The DLS construction is applied with a set of permutations defined by the $n$ rows of
the $n \times n$ Latin square $M$ with no $2 \times 2$ subsquares.
The DLS Construction can be applied also with Latin squares in which
there is no requirements for the nonexistence of $2 \times 2$ subsquares as was
done in~\cite{Etz96}. The constructions of large sets given in the sequel also
do not require arrays of permutations with no $2 \times 2$ subsquares.
The construction to achieve our goals will be applied with other sets (or multiset)
of permutations of $S_n$. In this case we have to make the following analysis.
Assume that $M$ is such a set (or multiset) of $\gamma \frac{(n-1)n}{2}$ permutations,
where $\gamma$ is even. $M$ can be viewed as a $( \gamma \frac{(n-1)n}{2} ) \times n$ matrix,
where each row represents a permutation from $M$.
Assume further that in each $( \gamma \frac{(n-1)n}{2} ) \times 2$ submatrix of $M$
each unordered pair $\{ i,j \}$, $i,j \in \Z_n$, appears
in exactly $\gamma$ rows. By applying these permutations of $M$ in the DLS Construction, instead
of the $n \times n$ Latin square, we obtain a set
of $\gamma \frac{(n-1)n}{2}$ SQS$(2n)$ which contains each quadruple from configuration $(3,1)$
in exactly $\frac{(n-1)\gamma}{2}$ systems and the same is true for each quadruple
from configuration $(1,3)$. Each quadruple
from configuration $(2,2)$ is contained in $\gamma$ of these systems. There are no quadruples from
configurations $(0,4)$ and $(4,0)$ in all these systems. These quadruples required in our main constructions can be obtained
with the next construction, namely the DB Construction. These observations will be used in
our main construction and are summarized as follows.
\begin{lemma}
\label{lem:sumDLS}
If the DLS Construction is applied with a perpendicular array PA$_\gamma (2,n,n)$, then
the DLS Construction yields a set $\cR$ with $\gamma \frac{(n-1)n}{2}$ systems (SQS$(2n)$)
on the point set $\Z_n \times \Z_2$ (with two natural parts $\Z_n\times\{0\}$ and $\Z_n\times\{1\}$).
\begin{enumerate}
\item Each quadruple from configuration $(3,1)$ is contained in exactly $\frac{\gamma (n-1)}{2}$ systems of $\cR$.

\item Each quadruple from configuration $(1,3)$ is contained in exactly $\frac{\gamma (n-1)}{2}$ systems of $\cR$.

\item Each quadruple from configuration $(2,2)$ is contained in exactly $\gamma$ systems of $\cR$.
\end{enumerate}
There are no quadruples from configurations $(4,0)$ and $(0,4)$ in any system of $\cR$.
\end{lemma}

\subsection{DB Construction (Doubling)}
\label{sec:DB}

Let $(\Z_n,B)$ be an SQS($n$), let $F=\{F_0,F_1,\ldots,F_{n-2} \}$ and
$F'=\{F'_0,F'_1,\ldots,F'_{n-2} \}$ be two one-factorizations (not necessarily distinct) of $K_n$ on the vertex set $\Z_n$
($F$ will be called the first one-factorization and $F'$ the second one-factorization).

Let $\alpha$ be any permutation on the set $\{ 0,1,\ldots,n-2 \}$. Define the collection of quadruples $B'$
on $\Z_n \times \Z_2$ as follows.

\begin{enumerate}
\item For each quadruple $\{ x_1,x_2,x_3,x_4 \} \in B$, the following two quadruples are
contained in~$B'$.
$$
\{ (x_1,0),(x_2,0),(x_3,0),(x_4,0) \},~~ \{ (x_1,1),(x_2,1),(x_3,1),(x_4,1) \}~.
$$

\item For each $i \in \{ 0,1,\ldots,n-2\}$ and $\{ x_1,x_2 \} \in F_i$, $\{ y_1,y_2 \} \in F'_j$,
where $j=\alpha (i)$, the following quadruple is contained in $B'$.
$$
\{ (x_1,0),(x_2,0),(y_1,1),(y_2,1) \}~.
$$
\end{enumerate}
Then, $B'$ forms an SQS$(2n)$.

Now, assume that instead of the set $B$, we apply the DB Construction with\linebreak
an LS$(3,4,n;g(n-1))$. Instead of one permutation $\alpha$, the construction is applied with
$n-1$ permutations on $\{ 0,1,\ldots,n-2 \}$ taken from an $(n-1) \times (n-1)$ Latin square $M$.
An LS$(3,4,n;g(n-1))$ contains $g(n-1)(n-3)$ systems of SQS$(n)$, which implies that
each permutation of $M$ is applied $g(n-3)$ times to obtain SQS($2n$) in the LS$(3,4,n;g(n-1))$.
The following theorem summarizes the configurations of the quadruples obtained
in this construction.

\begin{lemma}
\label{lem:sumDB}
Assume that the DB Construction is applied with an $(n-1) \times (n-1)$ Latin square $M$ and an LS$(3,4,n;g(n-1))$,
where each permutation of $M$ is applied with $g(n-3)$ systems of the LS$(3,4,n;g(n-1))$. Then
the DB Construction yields a set $\cR$ with $g (n-3)(n-1)$ systems (SQS$(2n)$)
on the point set $\Z_n \times \Z_2$.
\begin{enumerate}
\item Each quadruple from configuration $(4,0)$ is contained in exactly $g(n-1)$ systems of $\cR$.

\item Each quadruple from configuration $(0,4)$ is contained in exactly $g(n-1)$ systems of $\cR$.

\item Each quadruple from configuration $(2,2)$ is contained in exactly $g(n-3)$ systems of $\cR$.
\end{enumerate}
There are no quadruples from configurations $(3,1)$ and $(1,3)$ in any system of $\cR$.
\end{lemma}

By combining the DLS Construction (Lemma~\ref{lem:sumDLS}) and the DB Construction (Lemma~\ref{lem:sumDB})
we infer the following result which yield a doubling construction for large set with multiplicity

\begin{theorem}
\label{thm:sumDoubling}
If the DLS Construction is applied with a perpendicular array PA$_\gamma (2,n,n)$
and the DB Construction is applied with an LS$(3,4,n;g(n-1))$, where $\gamma =2g$, then the outcome is
an LS$(3,4,2n;g(n-1))$.
\end{theorem}
\begin{proof}
By Lemma~\ref{lem:sumDLS}, if the DLS Construction is applied with a PA$_\gamma (2,n,n)$, then each
quadruple from configuration (3,1) or configuration (1,3) is contained once in exactly ${\mu = \frac{\gamma (n-1)}{2}}$
systems. Each quadruple from configuration (2,2) is contained once in exactly $\gamma$ systems.
By Lemma~\ref{lem:sumDB} is applied with an LS$(3,4,n;\frac{\gamma (n-1)}{2})$, then each quadruple
from configuration (2,2) is contained in exactly $\frac{\gamma (n-3)}{2} =\mu -\gamma$ systems.
Finally, each quadruple from configuration (4,0) or configuration (0,4) is contained in exactly
$\mu$ of the systems. Thus, there exists an LS$(3,4,2n; \mu)$.
\end{proof}
Applying Theorem~\ref{thm:sumDoubling} with PA$_2(2,10,10)$ (see Theorem~\ref{thm:PA_pairs})
and LS$(3,4,10;9)$ (see Theorem~\ref{thm:LS10}) we have
\begin{corollary}
There exists an LS$(3,4,20;9m)$ for each $m \geq 1$.
\end{corollary}

\subsection{QLS Construction (Quadrupling)}
\label{sec:QLS}

The QLS Construction is a quadrupling construction, which was presented first by Lindner~\cite{Lin85}, while another variant was
introduced by Etzion and Hartman~\cite{EtHa91}. In these two papers the purpose of the construction was to obtain
pairwise disjoint Steiner quadruple systems. Recently, another simpler variant to obtain SQS$(4n)$ with good sequencing
was presented by Blackburn and Etzion~\cite{BlEt20}. In this section, the structure of the
systems obtained in this construction will be described. The exact formulation and the formal steps
of this construction, in the variant required for our constructions,
will be described in details when it will be used in Section~\ref{sec:fourn} and Section~\ref{sec:formal_steps}.

In the doubling construction, we consider only five configurations $(4,0)$,
$(3,1)$, $(2,2)$, $(1,3)$, and $(0,4)$. In the quadrupling construction we have to consider
thirty five configurations as follows.
The $4n$ points of $\Z_n \times \Z_4$ are partitioned into four equal parts, $\Z_n \times \{ i\}$, $i \in \Z_4$. The possible
configurations of quadruples are categorized into five groups.

\noindent
{\bf Group 1:} In this group there are four configurations $(4,0,0,0)$, $(0,4,0,0)$,
$(0,0,4,0)$, and $(0,0,0,4)$.

\noindent
{\bf Group 2:} In this group there are twelve configurations $(3,1,0,0)$, $(1,3,0,0)$, $(3,0,1,0)$, $(1,0,3,0)$, $(3,0,0,1)$,
$(1,0,0,3)$, $(0,3,1,0)$, $(0,1,3,0)$, $(0,3,0,1)$, $(0,1,0,3)$, $(0,0,3,1)$, and $(0,0,1,3)$.

\noindent
{\bf Group 3:} In this group we have the six configurations $(2,2,0,0)$, $(2,0,2,0)$, $(2,0,0,2)$,
$(0,2,2,0)$, $(0,2,0,2)$, and $(0,0,2,2)$.

\noindent
{\bf Group 4:} In this group there are twelve configurations $(2,1,1,0)$, $(2,1,0,1)$, $(2,0,1,1)$, $(1,2,1,0)$, $(1,2,0,1)$,
$(0,2,1,1)$, $(1,1,2,0)$, $(1,0,2,1)$, $(0,1,2,1)$, $(1,1,0,2)$, $(1,0,1,2)$, and $(0,1,1,2)$.

\noindent
{\bf Group 5:} In this group there is one configuration $(1,1,1,1)$.

The QLS Construction is based on quadruples from Groups 2, 3, 4, and 5.
We start with an SQS$(n)$ defined on $\Z_n$ and an $n \times n$ Latin square with no $2 \times 2$ subsquares.
The DLS Construction is applied to obtain a set $\cT$ of $n$ pairwise disjoint SQS$(2n)$.
We are now in the position to describe the framework of our variant of the quadrupling construction to construct $3n$ pairwise
disjoint SQS$(4n)$.
An SQS$(4n)$ in this construction is one of three types, Type A1, Type A2, and Type A3.
A set with $3 n^3$ quadruples from configuration $(1,1,1,1)$ is chosen, and partitioned
into three subsets of size $n^3$ with properties which will be defined in the sequel.
One subset will be in Type A1, one in Type A2, and one in Type A3, with each type consisting of $n$ SQS$(4n)$.
We will describe first the configurations for Type A1. Each SQS$(2n)$ from $\cT$ is embedded on the point set
$Z_n \times \{ 0,1 \}$ and on the point set $Z_n \times \{2,3\}$.
The two related Steiner quadruple systems $S(3,4,2n)$ are part of an SQS$(4n)$
Therefore, the SQS$(4n)$ contains quadruples
from configurations $(3,1,0,0)$, $(1,3,0,0)$,  $(0,0,3,1)$, and $(0,0,1,3)$.
It contains $\binom{n}{3}$ quadruples from each one of these four configurations.
It also contains quadruples from configurations $(2,2,0,0)$, and $(0,0,2,2)$,
$\binom{n}{2}$ quadruples from each configuration.
The SQS$(4n)$ of Type A1 also contains quadruples from configurations $(2,0,1,1)$, $(0,2,1,1)$, $(1,1,2,0)$,
and $(1,1,0,2)$. There are $\binom{n}{2} n$ quadruples in each one of these configurations.
The last configuration in Type A1 is $(1,1,1,1)$ and there are $n^2$ quadruples from this configuration
in each system.

Similarly an SQS$(4n)$ in Type A2 is constructed. It has quadruples from configurations
$(3,0,0,1)$, $(1,0,0,3)$,  $(0,1,3,0)$, $(0,3,1,0)$, $(2,0,0,2)$, $(0,2,2,0)$,
$(2,1,1,0)$, $(0,1,1,2)$, $(1,2,0,1)$, $(1,0,2,1)$, and $(1,1,1,1)$.
Type A3 has quadruples from configurations
$(3,0,1,0)$, $(1,0,3,0)$,  $(0,1,0,3)$, $(0,3,0,1)$, $(2,0,2,0)$, $(0,2,0,2)$,
$(2,1,0,1)$, $(0,1,2,1)$, $(1,2,1,0)$, $(1,0,1,2)$, and $(1,1,1,1)$.

This kind of quadrupling construction was described in~\cite{EtHa91,Lin85}.
In this paper, a perpendicular array PA$_\gamma (2,n,n)$ will be used as indicated in Lemma~\ref{lem:sumDLS}
instead of the $n \times n$ Latin square without no $2 \times 2$ subsquares. Such a perpendicular array
has $\gamma \binom{n}{2} = \frac{\gamma (n-1)}{2} n =\mu n$ permutations, i.e., the construction is applied
$\frac{\gamma (n-1)}{2}$ times compared to one application with an $n \times n$ Latin square.
To conclude, the total number of systems in Type A1 will be $\gamma \binom{n}{2}=\mu n$.
The same number of systems will be in Type A2 and in Type A3.

\section{LS$(3,4,4n;g)$, for $n \equiv 2$ or $4~(\text{mod}~6)$, $n \geq 10$}
\label{sec:five}

In this section we prove the first result which leads to the main result of this work.

\begin{theorem}
\label{thm:sum4n}
If there exists a large set LS$(3, 4, n;\mu)$ and a perpendicular
array PA$_{\gamma}(2,n,n)$, where $\mu=\frac{(n-1)\gamma}{2}$,
then there exists a large set LS$(3,4,4n;\mu)$.
\end{theorem}

Theorem~\ref{thm:sum4n}, which will be proved in this section, is a special case of the main result, and it will lead to
the main result of this work which will be proved in the next section.

\begin{theorem}
\label{thm:rec4n}
If there exists a large set LS$(3, 4, n;\mu)$ and a perpendicular
array PA$_{\gamma}(2, n,n)$, $\mu={(n-1)\gamma\over 2}$,
then there exists a large set LS$(3, 4, 2^m n;\mu)$, for each $m \geq 0$.
\end{theorem}

The first step in the proof of Theorem~\ref{thm:sum4n} is to apply the two doubling constructions,
the DLS Construction and the DB Construction, in the right combination to obtain
a large set of SQS($2n$) with multiplicity. This was summarized in Theorem~\ref{thm:sumDoubling}.
The rest of the proof of Theorem~\ref{thm:sum4n} will be based on the QLS Construction,
using the consequences of Lemma~\ref{lem:sumDLS} and Lemma~\ref{lem:sumDB} which are
partially summarized in Theorem~\ref{thm:sumDoubling}.
The quadrupling construction in this section, which is a variant of the QLS Construction,
will be called the $(4n)$-Construction. This Step of the construction is presented in Section~\ref{sec:fourn}
and Section~\ref{sec:formal_steps}. In Section~\ref{sec:fourn}, the ideas of the construction are
presented and in Section~\ref{sec:formal_steps} the formal definition for the blocks of the construction and the
proofs for the correctness of the construction, are given.
In these steps the QLS Construction is adapted to obtain a large set with multiplicity
which will be presented in the sequel.

\subsection{The $(4n)$-Construction - Introduction}
\label{sec:fourn}

Assume that on the point set $\Z_n$ there exists a large set LS$(3,4,n;\mu)$,
and a perpendicular array PA$_\gamma (2,n,n)$, where $\mu = \frac{(n-1)\gamma}{2}$ (which implies
that $\gamma$ is even). Note, that the number of permutations in
this perpendicular array is $n \mu = \frac{n(n-1)\gamma}{2}$.
Such a large set LS$(3,4,n;\mu)$ consists of $\mu (n-3)$ systems of SQS$(n)$.
In the recursive construction to form a large set LS$(3,4,n \cdot 2^m;\mu)$ presented in Section~\ref{sec:mainC},
the first step is a construction of a large set LS$(3,4,4n;\mu)$ which is presented
in this section.
Such a large set consists of $\mu(4n-3)$ systems of SQS$(4n)$.
We start by applying a variant of the QLS Construction to obtain $3 \mu n$ systems of SQS$(4n)$.
After these $3 \mu n$ systems of SQS$(4n)$ are obtained we continue with a variant of
the DB Construction to obtain $\mu (n-3)$ systems of SQS$(4n$).
This part of the construction is recursive and is based on another quadrupling construction
in which the large set LS$(3,4,n;\mu)$ is used.
This part will also use a one-factorization $F = \{ F_0 , F_1 , \ldots ,F_{n-2} \}$ of the complete
graph $K_n$ on vertex set $\Z_n$. We start with the construction of the large set
LS$(3,4,4n;\mu)$ on the point set $\Z_n \times \Z_4$. Let $M$ be the perpendicular array PA$_\gamma (2,n,n)$.
We apply the DLS Construction with any SQS$(n)$ and the $n \mu$ permutations of $M$ to obtain a set~$\cR$
with $n \mu =\frac{n(n-1)\gamma}{2}$ systems of SQS$(2n)$ on the point set $\Z_n \times \Z_2$.
Let $\cR_i$, $1 \leq i \leq \frac{n(n-1)\gamma}{2}$ be the $i$th such SQS$(2n)$ in $\cR$.
The $(4n)$-construction for the LS$(3,4,4n;\mu)$
has two types of SQS$(4n)$, Type A and Type B.

\noindent
{\bf Type A:}

In this type there are quadruples from the configurations in Groups 2, 3, 4, and 5.
The systems in Type A are of the three sub-types, Type~A1, Type A2, and Type A3, as described in Section~\ref{sec:QLS}.
Note that the total number of systems (SQS$(4n)$) of Type A1 is $n \mu =\gamma \binom{n}{2}$, which is
the same as the number of systems (SQS$(2n)$) in $\cR$.
The same number of systems are in Type A2 and the same number is also in Type A3. Thus, the total number of systems
in Type A is $3 n \mu$.

\noindent
{\bf Type B:}

Each system of Type B contains quadruples from Groups 1, 3, and 5.
It contains $n^3$ quadruples of configuration $(1,1,1,1)$,
$(n-1)\frac{n^2}{4}$ from each one of the configurations
$(2,2,0,0)$, $(2,0,2,0)$, $(2,0,0,2)$, $(0,2,2,0)$, $(0,2,0,2)$, and $(0,0,2,2)$.
It contains also $\binom{n}{3} {\Huge / } 4$ quadruple from each one of the configurations
$(4,0,0,0)$, $(0,4,0,0)$, $(0,0,4,0)$, and $(0,0,0,4)$.

The total number of systems in Type B will be $\gamma \frac{n-1}{2} (n-3)=(n-3) \mu$.

\noindent
{\large {\bf Summary:}}

To summarize, the total number of systems in Type A and Type B is $(4n-3) \mu$ which is the number of
systems required in a large set LS$(3,4,4n;\mu)$. In the next subsection the formal definition of
the blocks in these systems of all types are presented.

\subsection{The $(4n)$-Construction - Definitions and Proofs}
\label{sec:formal_steps}

In this subsection the formal definition for all the blocks in Type A and in Type B will be
presented. Proofs that the defined blocks yield an LS$(3,4,4n;\mu)$ are also given
in this subsection.
The blocks are formed on the point set $\Z_n \times \Z_4$; recall also that $\mu = \frac{\gamma (n-1)}{2}$.
For the construction, the following structures are required as input:
\begin{itemize}
\item Let $F= \{ F_0,F_1,\ldots,F_{n-2}\}$ be a one-factorization of $K_n$ on the point set $\Z_n$.

\item A set $\{ \cS_{(i,j)} ~:~ 1 \leq i \leq n,~ 0 \leq j \leq \mu-1 \} = \{ \cR_i ~:~ 1 \leq i \leq n \mu \}$
with $n \mu$ systems of SQS$(2n)$ obtained from the perpendicular array PA$_\gamma (2,n,n)$ via the DLS Construction.
\end{itemize}

\noindent
{\bf Type A1:}

Let $\cT_{(i,j)}$, be the $(i,j)$-th system of Type A1, $1 \leq i \leq n$, $0 \leq j \leq \mu-1$, which
is defined as follows.

The SQS($2n$) $\cS_{(i,j)}$, constructed on the point set $\Z_n \times \Z_2$,
is embedded on the point set $\Z_n \times \{0,1\}$ and on the point set $\Z_n \times \{2,3\}$.
These sets of quadruples contain quadruples from Group 2 and from Group 3.
They form the first set of quadruples which are constructed in $\cT_{(i,j)}$.

Next, we form the following four sets of quadruples from Group 4 in~$\cT_{(i,j)}$.
$$
\{ \{ (a,0),(b,0),(c,2),(c+i+r+j,3) \} ~:~ 0 \leq r \leq n-2,~ \{ a,b \} \in F_r,~ c \in \Z_n \},
$$
$$
\{ \{ (a,1),(b,1),(c,2),(c+i+r+j,3) \} ~:~ 0 \leq r \leq n-2,~ \{ a,b \} \in F_r,~ c \in \Z_n \},
$$
$$
\{ \{ (a,2),(b,2),(c,0),(c+i+r,1) \} ~:~ 0 \leq r \leq n-2,~ \{ a,b \} \in F_r,~ c \in \Z_n \},
$$
$$
\{ \{ (a,3),(b,3),(c,0),(c+i+r,1) \} ~:~ 0 \leq r \leq n-2,~ \{ a,b \} \in F_r,~ c \in \Z_n \}.
$$
These sets of quadruples form the second set of quadruples in $\cT_{(i,j)}$.

The last (third) set of quadruples in $\cT_{(i,j)}$ is from group 5:
$$
\{ \{ (a,0),(a+i+n-1,1),(b,2),(b+i+n-1+j,3) \} ~:~ a,b \in \Z_n  \}.
$$
\begin{lemma}
\label{lem:correctA1}
Each set $\cT_{(i,j)}$, $1 \leq i \leq n$, $0 \leq j \leq \mu-1$, is an SQS$(4n)$.
\end{lemma}
\begin{proof}
Since $\cS_{(i,j)}$ is an SQS($2n$), it follows that all triples from configurations
$(2,1,0,0)$, $(1,2,0,0)$, $(0,0,2,1)$, $(0,0,1,2)$, $(3,0,0,0)$, $(0,3,0,0)$,
$(0,0,3,0)$, $(0,0,0,3)$ are contained in quadruples of $\cT_{(i,j)}$.
Given a triple from configuration $(2,0,1,0)$, say $X=\{ (a,0),(b,0),(c,2)\}$, there is
a unique $r$, $0 \leq r \leq n-2$, such that $\{ a,b \} \in F_r$, and $X$ is contained in
the quadruple $\{ (a,0),(b,0),(c,2),(c+i+r+j,3) \}$ of $\cT_{(i,j)}$.
Similarly, each triple from configurations $(2,0,0,1)$, $(0,2,1,0)$, $(0,2,0,1)$, $(1,0,2,0)$,
$(0,1,2,0)$, $(1,0,0,2)$, and $(0,1,0,2)$, is contained in one of the quadruples of $\cT_{(i,j)}$.
Finally, consider a triple $X=\{ (c,0),(d,1),(b,2)\}$ from configuration $(1,1,1,0)$. Consider the quadruple
$Y=\{ (a,2),(b,2),(c,0),(c+i+r,1) \}$ of $\cT_{(i,j)}$. If $d=c+i+r$, where $0 \leq r \leq n-2$, then the
triple $X$ is contained in $Y$. If the solution for $d=c+i+r$ is only $r=n-1$, then $X$ is contained in the quadruple
$\{ (c,0),(c+i+n-1,1),(b,2),(b+i+n-1+j,3) \}$ of $\cT_{(i,j)}$.
Similarly, each triple from configurations $(1,1,0,1)$, $(1,0,1,1)$,
and $(0,1,1,1)$, is contained in one of the quadruples of~$\cT_{(i,j)}$.
Thus, each triple of the point set $\Z_n \times \Z_4$ is contained in some quadruple of~$\cT_{(i,j)}$.

The first set of quadruple contains $2 \binom{2n}{3}/4$ quadruples, the second set contains
$4 \binom{n}{2} n$ quadruples, and the third set $n^2$ quadruples.
Hence, the total number of quadruples in $\cT_{(i,j)}$ is $2 \binom{2n}{3}/4 + 4 \binom{n}{2} n + n^2 = \binom{4n}{3} / 4$.
Since this is the number of quadruples in an SQS($4n$) and each triple is contained in at least one quadruple,
it follows that each triple is contained in exactly one
quadruple of $\cT_{(i,j)}$, which completes the proof.
\end{proof}
\begin{lemma}
\label{lem:gr5A1}
The quadruples from configuration $(1,1,1,1)$ (Group 5)
which are contained in the $n \mu$ systems of SQS$(4n)$ from Type A1, form the
following $\mu$ sets $L_j$, $0 \leq j \leq \mu -1$, each one of size $n^3$:
$$
L_j \triangleq \{ \{ (x,0),(y,1),(z,2),(z+y-x+j,3)  \} ~:~ x,y,z \in \Z_n \}.
$$
\end{lemma}
\begin{proof}
By the definition of the quadruples from configuration $(1,1,1,1)$ in one system of Type~A1.
For a given $j$, $0 \leq j \leq \mu -1$, in all the related systems of Type A1, we have the set of quadruples
$$
\bigcup_{i=1}^n \{ \{ (a,0),(a+i+n-1,1),(b,2),(b+i+n-1+j,3) \} ~:~ a,b \in \Z_n  \}
$$
$$
= \bigcup_{i=1}^n \{ \{ (a,0),(y=a+i+n-1,1),(b,2),(b+y-a+j,3) \} ~:~ a,b \in \Z_n  \}
$$
$$
= \{ \{ (x,0),(y,1),(z,2),(z+y-x+j,3)  \} ~:~ x,y,z \in \Z_n \}~.
$$
Clearly, each such set is of size $n^3$ which completes the proof.
\end{proof}

\noindent
{\bf Type A2:}

Let $\cT'_{(i,j)}$, be the $(i,j)$-th system of Type A2, $1 \leq i \leq n$, $0 \leq j \leq \mu-1$, which
is defined as follows.

The SQS($2n$) $\cS_{(i,j)}$, constructed on the point set $\Z_n \times \Z_2$,
is embedded on the points set $\Z_n \times \{0,2\}$ and on the point set $\Z_n \times \{1,3\}$.
These sets of quadruples contain quadruples from Group 2 and from Group 3.
They form the first set of quadruples in $\cT'_{(i,j)}$.

Next, we form the following four sets of quadruples from Group 4 in~$\cT'_{(i,j)}$.
$$
\{ \{ (a,0),(b,0),(c,1),(c+i+r+j,3) \} ~:~ 0 \leq r \leq n-2,~ \{ a,b \} \in F_r,~ c \in \Z_n \},
$$
$$
\{ \{ (a,2),(b,2),(c,1),(c+i+r+j,3) \}~:~ 0 \leq r \leq n-2,~ \{ a,b \} \in F_r,~ c \in \Z_n \},
$$
$$
\{ (a,1),(b,1),(c,0),(c+i+r,2) \}~:~ 0 \leq r \leq n-2,~ \{ a,b \} \in F_r,~ c \in \Z_n \},
$$
$$
\{ \{ (a,3),(b,3),(c,0),(c+i+r,2) \}~:~ 0 \leq r \leq n-2,~ \{ a,b \} \in F_r,~ c \in \Z_n \}.
$$
These sets of quadruples form the second set of quadruples in $\cT'_{(i,j)}$.

The last (third) set of quadruples in $\cT'_{(i,j)}$ is from group 5.
$$
\{\{ (a,0),(b,1),(a+i+n-1,2),(b+i+n-1+j,3)\} ~:~ a,b \in \Z_n  \}.
$$
Similarly, to Lemma~\ref{lem:correctA1} and Lemma~\ref{lem:gr5A1} we prove the following two lemmas.
\begin{lemma}
\label{lem:correctA2}
Each set $\cT'_{(i,j)}$, $1 \leq i \leq n$, $0 \leq j \leq \mu-1$, is an SQS$(4n)$.
\end{lemma}
\begin{lemma}
\label{lem:gr5A2}
The quadruples from configuration $(1,1,1,1)$ (Group 5)
which are contained in the $n \mu$ systems of SQS$(4n)$ from Type A2, form the
following $\mu$ sets $L'_j$, $0 \leq j \leq \mu -1$, each one of size $n^3$:
$$
L'_j \triangleq \{ \{ (x,0),(y,1),(z,2),(y+z-x+j,3)  \} ~:~ x,y,z \in \Z_n \}.
$$
\end{lemma}

\noindent
{\bf Type A3:}

Let $\cT''_{(i,j)}$, be the $(i,j)$-th system of Type A3, $1 \leq i \leq n$, $0 \leq j \leq \mu-1$, which
is defined as follows.

The SQS($2n$) $\cS_{(i,j)}$, constructed on the point set $\Z_n \times \Z_2$,
is embedded on the points set $\Z_n \times \{0,3\}$ and on the point set $\Z_n \times \{1,2\}$.
These set of quadruples contains quadruples from Group 2 and from Group 3.
They form the first set of quadruples in $\cT''_{(i,j)}$.

Next, we form the following four sets of quadruples from Group 4 in~$\cT''_{(i,j)}$.
$$
\{ \{ (a,0),(b,0),(c,1),(c+i+r,2) \} ~:~ 0 \leq r \leq n-2,~ \{ a,b \} \in F_r,~ c \in \Z_n \},
$$
$$
\{ \{ (a,3),(b,3),(c,1),(c+i+r,2) \} ~:~ 0 \leq r \leq n-2,~ \{ a,b \} \in F_r,~ c \in \Z_n \},
$$
$$
\{ \{ (a,1),(b,1),(c,0),(c+i+r+j,3) \} ~:~ 0 \leq r \leq n-2,~ \{ a,b \} \in F_r,~ c \in \Z_n \},
$$
$$
\{ \{ (a,2),(b,2),(c,0),(c+i+r+j,3) \} ~:~ 0 \leq r \leq n-2,~ \{ a,b \} \in F_r,~ c \in \Z_n \}.
$$
These sets of quadruples form the second set of quadruples in $\cT''_{(i,j)}$.

The last (third) set of quadruples in $\cT''_{(i,j)}$ is from group 5.
$$
\{\{ (a,0),(b,1),(b+i+n-1,2),(a+i+n-1+j,3)\} ~:~ a,b \in \Z_n  \}.
$$
Similarly, to Lemma~\ref{lem:correctA1} and Lemma~\ref{lem:gr5A1} we prove the following two lemmas.
\begin{lemma}
\label{lem:correctA3}
Each set $\cT''_{(i,j)}$, $1 \leq i \leq n$, $0 \leq j \leq \mu-1$, is an SQS$(4n)$.
\end{lemma}
\begin{lemma}
\label{lem:gr5A3}
The quadruples from configuration $(1,1,1,1)$ (Group 5)
which are contained in the $n \mu$ systems of SQS$(4n)$ from Type A3, form the
following $\mu$ sets $L''_j$, $0 \leq j \leq \mu -1$, each one of size $n^3$:
$$
L''_j \triangleq \{ \{ (x,0),(y,1),(z,2),(x+z-y+j,3)  \} ~:~ x,y,z \in \Z_n \}
$$
\end{lemma}

The next step for Type A is to calculate the number of times that each quadruple from each configuration
is contained in the $3 \mu n$ SQS($4n$) of Type A. Recall, for the next theorem, that in the SQS($2n$) $\cS_{(i,j)}$,
$1 \leq i \leq n$, $0 \leq j \leq \mu-1$, is constructed in the DLS Construction, by
using a PA$_\gamma (2,n,n)$, where $\gamma = \frac{2 \mu}{n-1}$, i.e. $\mu = \frac{(n-1)\gamma}{2}$.

\begin{lemma}
\label{lem:propA}
The systems of Type A in the $(4n)$-construction have the following containment properties.
\begin{enumerate}
\item[(1)] Each quadruple from each configuration of Group 2
is contained in exactly $\mu = \frac{(n-1)\gamma}{2}$ systems of Type A.

\item[(2)] Each quadruple from Group 3 is contained in exactly $\gamma = \frac{2 \mu}{n-1}$ systems
of Type A.

\item[(3)] Each quadruple from Group 4 is contained in exactly $\mu$ systems.

\item[(4)] The total number of quadruples from configuration $(1,1,1,1)$,
which are contained in the systems of Type A is $3 \mu n^3$. These quadruples are the $3 \mu n^3$ quadruples
defined in the $3 \mu$ sets $L_j$, $L'_j$, and $L''_j$, $0 \leq j \leq \mu-1$.
\end{enumerate}
\end{lemma}
\begin{proof}
$~$
\begin{enumerate}
\item[(1)] is an immediate consequence from Lemma~\ref{lem:sumDLS} and the definitions
of quadruples of Type A1, Type A2, and Type A3.

\item[(2)] is also an immediate consequence from Lemma~\ref{lem:sumDLS} and the definitions
of quadruples of Type A1, Type A2, and Type A3.

\item[(3)] By the definition of Group 4, the total number of quadruples in a given configuration of Group 4
is $\frac{n(n-1)}{2} n^2$. Each SQS$(4n)$ of Type A contains $4 \frac{n(n-1)}{2} n$ such
quadruples from four distinct configurations, $\frac{n(n-1)}{2} n$ quadruples from each configuration.
Clearly, there is no intersection between the configurations used in Type A1 to those
used in Type A2 (and similarly between Type A1 and Type A3, and between Type A2 and Type A3).
In each such type there are four configurations from the twelve configurations of Group 4. For example,
in Type A1 there are quadruples from configurations $(2,0,1,1)$, $(0,2,1,1)$, $(1,1,2,0)$, and $(1,1,0,2)$.
Consider for example, the quadruples from type A1 for fixed $i$ and $j$,
\begin{equation}
\label{eq:quad1}
\{ (a,0),(b,0),(c,2),(c+i+r+j,3) \},
\end{equation}
for each $r$, $0 \leq r \leq n-2$, $\{ a,b \} \in F_r$, and $c \in \Z_n$.
There are $(n-1) n$ distinct ways to choose a pair $(r,c)$ and $\frac{n}{2}$ pairs from $F_r$.
Hence, Type A1 contains exactly $\frac{(n-1)n^2}{2}$ quadruples from configuration $(2,0,1,1)$.
The same calculation holds for each configuration in Type A1, Type A2, and Type A3.
Clearly for $1 \leq i_1 < i_2 \leq n$, the quadruples in (\ref{eq:quad1}) for $i=i_1$ and for $i=i_2$
are distinct which implies that each quadruple from configuration $(2,0,1,1)$ is contained exactly once in Type A1 for a fixed $j$.
Thus, in the $3 n \mu$ systems of Type A,
each quadruple from Group 4 is contained in exactly $\mu$ systems.

\item[(4)] is an immediate consequence from Lemma~\ref{lem:gr5A1}, Lemma~\ref{lem:gr5A2}, and Lemma~\ref{lem:gr5A3}.
\end{enumerate}
\end{proof}

\vspace{0.5cm}

\noindent
{\bf Type B:}

For the systems of Type B, to be defined in the $(4n)$-construction,
the following structures are required as input:

\begin{itemize}
\item Let $F= \{ F_0,F_1,\ldots,F_{n-2}\}$ be a one-factorization of $K_n$ on the point set $\Z_n$.

\item Let $M$ be an $(n-1) \times (n-1)$ Latin square on the points set $\{ 0,1,\ldots,n-2 \}$.

\item A set $\{ \cS^*_i ~:~ 1 \leq i \leq \mu (n-3)\}$, on the point set $\Z_n$, which form an LS$(3,4,n;\mu)$.
\end{itemize}

Let $\cP_i$, $1 \leq i \leq \mu (n-3)$, be the $i$-th system of Type B. Its blocks are defined
on the point set $\Z_n \times \Z_4$ as follows.

$$
\{ \{ (x_1,j),(x_2,j),(x_3,j),(x_4,j) \} ~:~ \{ x_1,x_2,x_3,x_4 \} \in \cS^*_i, ~ j \in \Z_4 \}.
$$

Given $j$, $0 \leq j \leq n-2$,
for $\cP_i$, $j (\mu - \gamma) < i \leq (j+1) (\mu - \gamma)$, the following six blocks from Group 3
are defined.
$$
\{ \{ (x,0),(y,0),(z,1),(v,1)  \} ~:~ \{x,y\} \in F_r,~ \{z,v\} \in F_{M(j,r)}, ~ 0 \leq r \leq n-2 \},
$$
$$
\{ \{ (x,0),(y,0),(z,2),(v,2)  \} ~:~ \{x,y\} \in F_r,~ \{z,v\} \in F_{M(j,r)}, ~ 0 \leq r \leq n-2 \},
$$
$$
\{ \{ (x,0),(y,0),(z,3),(v,3)  \} ~:~ \{x,y\} \in F_r,~ \{z,v\} \in F_{M(j,r)},~ 0 \leq r \leq n-2 \},
$$
$$
\{ \{ (x,1),(y,1),(z,2),(v,2)  \} ~:~ \{x,y\} \in F_r,~ \{z,v\} \in F_{M(j,r)}, ~ 0 \leq r \leq n-2 \},
$$
$$
\{ \{ (x,1),(y,1),(z,3),(v,3)  \} ~:~ \{x,y\} \in F_r,~ \{z,v\} \in F_{M(j,r)}, ~ 0 \leq r \leq n-2 \},
$$
$$
\{ \{ (x,2),(y,2),(z,3),(v,3)  \} ~:~ \{x,y\} \in F_r,~ \{z,v\} \in F_{M(j,r)}, ~ 0 \leq r \leq n-2 \}.
$$

\begin{remark}
Note, that $(n-1) (\mu - \gamma) = (n-3) \mu$ and hence there is no ambiguity in the definition
of the $\cP_i$'s.
\end{remark}

The most challenging design part, of Type B, is to construct the quadruples
from configuration $(1,1,1,1)$ to accommodate the large set. There are $n^4$ quadruples from configuration $(1,1,1,1)$,
each one should be contained in exactly $\mu$ of the $\mu (4n-3)$ systems of the
LS$(3,4,4n;\mu)$ which is constructed. The sets of quadruples
from configuration $(1,1,1,1)$ for our system of Type B are defined as follows.

Let $t$ be the smallest positive integer such that $t n \geq \mu$. For each $j$, $0 \leq j \leq tn -1$,
we form the set
$$
D_j \triangleq \{ \{ (x,0),(y,1),(z,2),(x+z-y+j,3)  \} ~:~ x,y,z \in \Z_n \}
$$

For each $j$, $0 \leq j \leq (\mu -t)n-1$, we form the set
$$
E_j \triangleq \{ \{ (x,0),(y,1),(z,2),(z+y-x+j,3)  \} ~:~ x,y,z \in \Z_n \}
$$
Note, that two $D_j$'s are either the same or disjoint and the same is for the $E_j$'s.
Moreover, the $D_j$'s and the $E_j$'s might have nonempty intersection. Nevertheless, in the sequel
if $i \neq j$ then $D_i$ and $D_j$ will be considered as distinct sets. The same is true for $E_i$ and $E_j$.

\begin{lemma}
\label{lem:DEL1}
Each $D_j$ and each $E_j$ contains exactly $n^3$ quadruples from configuration $(1,1,1,1)$.
Each quadruple from configuration $(1,1,1,1)$ is contained in exactly $\mu$
of these $D_j$'s and $E_j$'s of the
set $\{ D_j ~:~ 0 \leq j \leq tn -1 \} \cup \{ E_j ~:~ 0 \leq j \leq (\mu -t)n -1 \}$.
\end{lemma}
\begin{proof}
There are $n^3$ distinct ways to choose $x,y,z \in \Z_n$ and hence each $D_j$ and
each $E_j$ contains exactly $n^3$ quadruples from configuration $(1,1,1,1)$.
Moreover, each such quadruple is an element in one of the $D_j$'s and one of the $E_j$'s.
The set $\{ D_j ~:~ 0 \leq j \leq tn -1 \}$ contains $tn$ subsets of quadruples, where
$D_j=D_{n+j}$ for $0 \leq j \leq (t-1)n-1$. Therefore, each quadruple from configuration (1,1,1,1)
is contained in exactly $t$ of the $D_j$'s.
The set $\{ E_j ~:~ 0 \leq j \leq (\mu -t)n -1 \}$ contains $(\mu -t)n$ subsets of quadruples, where
$E_j=E_{n+j}$ for $0 \leq j \leq (\mu -t-1)n -1$. Therefore, each quadruple from configuration (1,1,1,1)
is contained in exactly $\mu-t$ of the $E_j$'s. This completes the proof.
\end{proof}

\begin{lemma}
\label{lem:DEL2}
$~$
$$
\{ L_j,~ L'_j,~ L''_j ~:~ 0 \leq j \leq \mu-1 \}  \subset
\{ D_j ~:~ 0 \leq j \leq tn -1 \} \cup \{ E_j ~:~ 0 \leq j \leq (\mu -t)n -1 \}.
$$
\end{lemma}
\begin{proof}
By the definition $L_j = L'_j = E_j$ for each $0 \leq j \leq \mu-1$ and
$L''_j=D_j$ for each $0 \leq j \leq \mu-1$.
\end{proof}

\begin{corollary}
\label{cor:cof1111}
The number of subsets with $n^3$ quadruples from configuration $(1,1,1,1)$ in the multiset
$\{ D_j ~:~ 0 \leq j \leq tn -1 \} \cup \{ E_j ~:~ 0 \leq j \leq (\mu -t)n -1 \} \setminus \{ L_j,~ L'_j,~ L''_j ~:~ 0 \leq j \leq \mu-1 \}$
is $(n  -3) \mu$.
\end{corollary}

By Corollary~\ref{cor:cof1111} we have a partition of the quadruples (with repetitions) from
configuration $(1,1,1,1)$ which are not contained in Type A into $(n -3)\mu$ subsets, each one
of size $n^3$. These $(n -3)\mu$ subsets are distributed arbitrarily among the $\cP_i$'s
to complete the quadruples from Group 5 of the $\cP_i$'s.

\begin{lemma}
\label{lem:sysB}
Each system of Type B is an SQS$(4n)$.
\end{lemma}
\begin{proof}
One can easily verify that each triple from $\Z_n \times \Z_4$ is contained in one
of the blocks for each system of type B. Hence, to complete the proof it is sufficient to
show that the number of blocks in a system of Type B is $\binom{4n}{3} {\huge /} 4$ as required for
SQS$(4n)$. By their definition, in Type B, there are $\binom{n}{3}$ quadruples of Group 1 in each system.
The number of quadruple in each system from Group 3 is $6 (n-1)\frac{n^2}{4}$ and from Group 5 this number is $n^3$.
Since $\binom{n}{3} + 6 (n-1)\frac{n^2}{4} +n^3 = \binom{4n}{3} {\huge /} 4$, it follows that each system of Type B
is an SQS$(4n)$.
\end{proof}

\begin{lemma}
\label{lem:propB}
The systems of Type B in the $(4n)$-construction have the following containment properties.
\begin{enumerate}

\item[(1)] Each quadruple from each configuration of Group 1
is contained in exactly $\mu = \frac{(n-1)\gamma}{2}$ of the $(n-3)\mu$ systems of Type B.

\item[(2)] Each quadruple from Group 3 is contained in exactly $\mu - \gamma$ systems
of Type B.

\item[(3)] The total number of quadruples from configuration $(1,1,1,1)$ (Group 5),
which are contained in the systems of Type B is $(n -3)\mu n^3$.
\end{enumerate}
\end{lemma}
\begin{proof}
The enumeration is a straightforward result from the definitions.
\begin{enumerate}

\item[(1)] follows from the immediate observation that each quadruple from $\Z_n$
is contained in exactly $\mu$ systems of the LS$(3,4,n;\mu)$,
$\{ \cS^*_i ~:~ 1 \leq i \leq (n-3)\mu \}$.

\item[(2)] follows immediately from the definition that for each $j$, $0 \leq j \leq n-2$,
for each $i$, $j (\mu - \gamma) < i \leq (j+1) (\mu - \gamma)$, the related $\cP_i$'s contain
the same quadruples from Group~3.

\item[(3)] The number of $\cP_i$'s is $(n -3)\mu$. Each one contains either one of the $D_j$'s
or one of the $E_j$'s, where each one contains $n^3$ quadruples. Thus,
The total number of quadruples from configuration $(1,1,1,1)$,
which are contained in the systems of Type B is ${(n -3)\mu n^3}$.
\end{enumerate}
\end{proof}

\noindent
{\bf Proof of Theorem~\ref{thm:sum4n}:}

There are no quadruples from configurations of Group 1 in Type A, while
by Lemma~\ref{lem:propB} each quadruple from configurations of Group 1
is contained in exactly $\mu$ systems or Type B.

By Lemma~\ref{lem:propA} each quadruple from each configuration of Group 2
or Group 4 is contained in exactly $\mu$ systems of Type A, while in Type B
there are no quadruples from these groups.

By Lemma~\ref{lem:propA} each quadruple from each configuration of Group 3 is
contained in $\gamma$ systems of Type A and by Lemma~\ref{lem:propB}
each quadruple from each configuration of Group 3 is
contained in $\mu - \gamma$ systems of Type B. Thus, each quadruple from each configuration of Group 3
is contained in exactly $\mu$ systems of Type A or Type B.

By Lemma~\ref{lem:propA}, Lemma~\ref{lem:DEL1}, Lemma~\ref{lem:DEL2},
and Corollary~\ref{cor:cof1111}, each quadruple from configuration (1,1,1,1)
is contained in exactly $\mu$ systems of Type A or Type B.

This complete the proof of the theorem.

\begin{flushright}
$\Box$
\end{flushright}

\section{LS$(3,4,2^m n;g)$, for $n \equiv 2$ or $4~(\text{mod}~6)$, $m \geq 3$}
\label{sec:mainC}

The quadrupling construction of Section~\ref{sec:five} implies an LS$(3,4,4n;\mu)$
from an LS$(3,4,n;\mu)$. The goal in this section is to continue and prove Theorem~\ref{thm:rec4n}.
Given the LS$(3,4,4n;\mu)$ constructed by the $(4n)$-construction we amend it to
a construction for $LS (3,4,2^m n;\mu)$, where $m \geq 3$. First, note that in order
to apply the $(4n)$-construction recursively we need an appropriate perpendicular
array and it might not exist (in fact it probably does not exist for most parameters).

The construction which will be used here is based on the idea given in~\cite{EtHa91}.
The first $(2^m-1)n\mu$ systems are based on the first $3\mu n$ systems of Type A and
are constructed similarly to the systems as explained in~\cite{EtHa91}.

First, an order of $\Z_2^m$ is induced by identifying $x \in \Z_2^m$ with a nonnegative integer
smaller than $2^m$ whose binary representation is $x$. This is the usual lexicographic order.
Next, a set of $2^m-1$ SQSs on the
point set $\Z_2^m$ and block set $B_i$ with $i \in \Z_2^m \setminus \{ 0 \}$ is defined. These sets
are called the Boolean Steiner quadruple systems. The block set $B_i$, $i \in \Z_2^m \setminus \{0\}$
is defined to be the union of the blocks of Types (B.1) and (B.2) specified below
$$
\text{(B.1)} ~~~~ \{ \{ x,y,z,w \} ~:~ x+y+z+w = i, ~~ | \{x,y,z,w\} | =4 \}
$$
$$
\text{(B.2)} ~~~~ \{ \{ x,y,z,w \} ~:~ x+y=z+w = i, ~~ | \{x,y,z,w\} | =4 \}
$$

The following result which can be easily verified was proved in~\cite{EtHa91}.
\begin{lemma}
For each $i$, $i \in \Z_2^m \setminus \{0\}$ the set $B_i$ is an SQS$(2^m)$.
\begin{itemize}
\item Each quadruple of $\Z_2^m$ is contained in at least one of these $2^m-1$ SQS$(2^m)$.

\item Each block of Type (B.1) is contained in exactly one of the $2^m-1$ SQS$(2^m)$.

\item Each block of Type (B.2) is contained in exactly three of the $2^m-1$ SQS$(2^m)$.
\end{itemize}
\end{lemma}

We continue to define the blocks of each system in the $(2^m n)$-Construction.
The point set of each SQS$(2^m n)$ is $\Z_n \times Z_2^m$.
The first $(2^m-1)n\mu$ systems are based on the first $3\mu n$ systems of Type A and
are constructed similarly to the systems constructed in~\cite{EtHa91}.

The $(i,j)$-th system $\cP_{(i,j)}$, $i \in \Z_2^m \setminus \{0\}$, $0 \leq j \leq n\mu -1$, is defined as follows.
Recall that $\cR_s$, $1 \leq s \leq n \mu$, is a set of SQS$(2n)$ on the point set
$\Z_n \times \Z_2$ defined via the DLS Construction using a PA$_\gamma (2,n,n)$.
This set will be defined now on the point set $\Z_n \times \{x,y\}$, where $x,y \in \Z_2^m$ and $x+y = i$.

Now, assume that $x+y+z+w=0$, where $|\{x,y,z,w\}|=4$ and
$x,y,z,w \in \Z_2^m$, where $x<y,z<w$, $x+y=i$, $x+z=r$, $x+w=\ell$.
Let $g$ be a bijection $g: \Z_4 \rightarrow \{x,y,z,w\}$.
On the point set $Z_n \times \{g(0),g(1),g(2),g(3)\}$ we embed the $n\mu$ systems of Type A1 for $\cP_{(i,j)}$, $0 \leq j \leq n\mu -1$.
On the point set $Z_n \times \{g(0),g(1),g(2),g(3)\}$ we embed the $n \mu$ systems of Type A2 for $\cP_{(r,j)}$, $0 \leq j \leq n\mu -1$.
On the point set $Z_n \times \{g(0),g(1),g(2),g(3)\}$ we embed the $n \mu$ systems of Type A3 for $\cP_{(\ell,j)}$, $0 \leq j \leq n\mu -1$.

\begin{remark}
Note, that $\cR_s$ is embedded only once for each point set $\Z_n \times \{ x,y \}$ for each $x,y \in \Z_2^m$
such that $x+y=i$. Similarly, the quadruples from Group 4 and the quadruples from configuration $(1,1,1,1)$ are
embedded only once on the point set $\Z_n \times \{x,y,z,w\}$ for each $x,y,z,w \in \Z_2^m$ such that $x+y=z+w=i$.
\end{remark}

Assume now, that $x+y+z+w=i$, where $|\{x,y,z,w\}|=4$ and
$x,y,z,w \in \Z_2^m$. We form the following last set of blocks in $\cP_{(i,r)}$, $0 \leq r \leq n\mu -1$.
$$
\{ \{ (a,x),(b,y),(c,z),(a+b+c+r,w) \} ~:~ a,b,c \in \Z_n \}~.
$$

\begin{lemma}
\label{lem:lastA}
Each $\cP_{(i,j)}$ is an SQS$(2^m n)$.
\end{lemma}
\begin{proof}
It is straightforward to prove by the definition based on the DLS Construction and the $(4n)$-construction
that each triple of $\Z_n \times \Z_2^m$ is contained in at most one block.
Hence, to complete the proof it is sufficient to prove that each system $\cP_{(i,j)}$, $i \in \Z_2^m \setminus \{0\}$, $0 \leq j \leq n\mu -1$,
contains $\binom{2^m n}{3} {\huge /} 4$ blocks as is the number of blocks in an SQS$(2^m n)$.
The number of blocks originated from Type A in a system is as follows. There are $2^{m-1} \binom{2n}{3} {\huge /} 4$ blocks
for all SQS$(2n)$ embedded on the set of points $\Z_n \times \{x,y\}$ such that $x+y=i$. The number of blocks induced
from Type A from Group 4 or configuration (1,1,1,1) in a system is $\binom{2^{m-1}}{2} (4 \cdot (n-1) \frac{n}{2} n + n^2)$.
For each $x,y,z,w \in \Z_2^m$ such that $x+y+z+w=i$ and $|\{x,y,z,w\}|=4$ the number of block in the last set is $n^3$
and in all such $x,y,z,w$ the number of blocks is $\left(  \binom{2^m}{3} {\huge /} 4 - \binom{2^{m-1}}{2}  \right) n^3$.
Therefore, the total number of blocks in $\cP_{(i,j)}$ is
$$
2^{m-1} \binom{2n}{3} {\huge /} 4 + \binom{2^{m-1}}{2} (4 \cdot (n-1) \frac{n}{2} n + n^2) + \left(  \binom{2^m}{3} {\huge /} 4 - \binom{2^{m-1}}{2}  \right) n^3
$$
which equals $\binom{2^m n}{3} {\huge /} 4$ as required.
\end{proof}

Now, we define the last $(n-3)\mu$ systems.
For each quadruple $\{x,y,z,v\} \subset \Z_2^m$ such that $x+y+z+v=0$ we form the last
$(n-1) (\mu - \gamma) = (n-3) \mu$ systems from the $(4n)$-construction on the point set $\Z_n \times  \{x,y,z,v\}$.
These last $(n - 3) \mu$ systems (SQS$(2^m n)$) are constructed based on the $(n - 3) \mu$ systems (SQS$(4n)$)
of Type B constructed in the $(4n)$-construction. Let $\cP_i$ be the $i$-th (SQS$(4n)$), $1 \leq i \leq (n - 3) \mu$
of Type B. The first set of blocks of the $i$-th last system ($(2^m-1)n\mu+i$ system of the whole large set)
for the $(2^m n)$-construction are defined as follows.
Recall that $F=\{F_0,F_1, \ldots, F_{n-2} \}$ is a one-factorization on $K_n$
and $M$ is an $(n-1) \times (n-1)$ Latin square on the points set $\{ 0,1,\ldots,n-2 \}$.
Finally, $\cS^*_i$, $1 \leq i \leq (n-3) \mu$, is the $i$-th SQS$(n)$ in a large set LS$(3,4,n;\mu)$.

For each $j \in \Z_2^m$ and each block $\{ a,b,c,d \} \in \cS^*_i$ construct the block
$$
\{ (a,j),(b,j),(c,j),(d,j) \}.
$$

Given $j$, $0 \leq j \leq n-2$, and $i$,
$j (\mu - \gamma) < i \leq (j+1) (\mu - \gamma)$, the following blocks form the second set
of blocks in the $i$-th system (from the $(n-3) \mu$ last systems).

$$
\{ \{ (a,\ell),(b,\ell),(c,t),(d,t)  \} ~:~ \ell,t \in \Z_2^m,~ \ell < t,~ \{a,b\} \in F_r,~ \{c,d\} \in F_{M(j,r)},~ 0 \leq r \leq n-2 \}.
$$

For each block $\{ (a,0),(b,1),(c,2),(d,3) \} \in \cP_i$ and for each four distinct values
${x,y,z,w \in \Z_2^m}$ such that $x<y<z<w$ and $x+y+z+w=0$ construct the block
$$
\{ (a,x),(b,y),(c,z),(d,w) \}.
$$
These blocks form the third set of blocks in the construction.

\begin{lemma}
\label{lem:lastB}
Each one of the last $(n-3) \mu$ system is an SQS$(2^m n)$.
\end{lemma}
\begin{proof}
The number of blocks in the first set of each system is $2^m \binom{n}{3} {\huge /} 4$,
in the second set is $\binom{2^m}{2} (n-1) \frac{n^2}{4}$, and in the third set is $\binom{2^m}{3} n^3 {\huge /} 4$.
Hence, the total number of blocks in each system is $\binom{2^m n}{3} {\huge /} 4$ which is the number of blocks in an SQS$(2^m n)$.
One can easily continue and verify that no triple of $\Z_n \times \Z_2^m$ is contained in more than
one block of the system, which completes the proof.
\end{proof}

The proof of the following lemma is identical and follows from the proof of Lemma~\ref{lem:propB}.
\begin{lemma}
The last $(n - 3) \mu$ systems of SQS$(2^m n)$ have the following properties
\begin{enumerate}
\item Each quadruple $\{ (a,j),(b,j),(c,j),(d,j) \}$, $\{ a,b,c,d \} \subset \Z_n$, $j \in \Z_2^m$
is contained in exactly $\mu$ systems.

\item Each quadruple $\{ (a,i),(b,i),(c,j),(d,j) \}$, where $i,j \in \Z_2^m$, $i < j$,
and $\{a,b\} \subset \Z_n$, $\{ c,d \} \subset \Z_n$, is contained in exactly $\mu - \gamma$ systems.

\item For given four distinct values $x,y,z,w \in \Z_2^m$, such that $x+y+z+w=0$,
the number of quadruples of the form $\{ (a,x),(b,y),(c,z),(d,w) \}$, where
$a,b,c,d \in \Z_n$, which are contained in the last $(n-3) \mu$ systems
is $(n -3)\mu n^3$.
\end{enumerate}
\end{lemma}

\noindent
{\bf Proof of Theorem~\ref{thm:rec4n}:}

The existence of an LS$(3, 4, n;\mu)$ is given in the theorem. Hence, by Theorem~\ref{thm:sumDoubling}
there exists an LS$(3, 4, 2n;\mu)$ and by Theorem~\ref{thm:sum4n} there exists an LS$(3, 4, 4n;\mu)$.
The rest of the proof is induced from the $(2^m n)$-construction.

By Lemma~\ref{lem:lastA} and Lemma~\ref{lem:lastB}, each one of the systems of
the $(2^m n)$-Construction is an SQS$(2^m n)$. The number of such systems in the construction
is $(2^m -1) n \mu + (n-3) \mu = 2^m n \mu -3\mu$ as required in an LS$(3, 4, 2^m n;\mu)$, $m \geq 3$.

The proof that each quadruple of $\Z_n \times \Z_2^m$ is contained in exactly $\mu$ of the systems
is similar to the one in Theorem~\ref{thm:sum4n} and it is left to the reader.
\begin{flushright}
$\Box$
\end{flushright}

Theorem~\ref{thm:PA_pairs}, Theorem~\ref{thm:LS10}, and Theorem~\ref{thm:rec4n} imply that
\begin{corollary}
There exist an LS$(3,4,5 \cdot 2^m;9)$ for each integer $m \geq 1$.
\end{corollary}
\begin{corollary}
There exist an LS$(3,4,5 \cdot 2^m;9\ell)$ for each integer $m \geq 1$ and each integer $\ell \geq 1$.
\end{corollary}

Theorem~\ref{thm:rec4n} and Theorem~\ref{thm:expand} imply our final result.

\begin{corollary}
For each $m \geq 1$ and each $\ell \geq 1$ there exists an LH$(5 \cdot 2^m,9 \ell ,4,3)$.
\end{corollary}

\section{Conclusion and Problems for Future research}
\label{sec:conclusion}

The lack of known constructions for large sets of Steiner systems S$(t,k,n)$,
where $2 < t < k<n$ has motivated the definition of a large set of Steiner systems
with multiplicity. In such a system each $t$-subset of the $n$-set is contained in exactly
$\mu$ systems of the large set. The existence of such large sets with multiplicity implies
the existence of large sets for H-designs with related parameters. A recursive construction for
large set of Steiner quadruple systems with multiplicity was given. For small parameters some ad-hoc
constructions for large sets with multiplicity were given using perpendicular arrays and ordered
designs. Except for the large sets of H-designs derived from large sets with multiplicity, some
ad-hoc constructions for large sets of H-designs with blocks of size four and small number of groups
were also presented.

The exposition of this paper raises many open problems and in particular one would like to
see construction of large sets for larger range of parameters and small multiplicity as possible.



\end{document}